\documentclass[leqno]{amsart}
\usepackage{amssymb}
\usepackage{amsmath, amsfonts, vmargin, enumerate}
\usepackage{color}
\usepackage{amsthm}
\usepackage[colorlinks=true,citecolor=blue,linkcolor=blue]{hyperref}

\newtheorem{thm}{Theorem}[section]
\newtheorem{prop}[thm]{Proposition}
\newtheorem{cor}[thm]{Corollary}
\newtheorem{lem}[thm]{Lemma}
\newtheorem{defi}[thm]{Definition}
\newtheorem{remark}[thm]{Remark}
\newtheorem{example}[thm]{Example}

\numberwithin{equation}{section}

\newcommand{\Bc}{\mathcal{B}}

\newcommand{\D}{{\mathcal{D}}}
\newcommand{\sym}{{\mathrm{sym}}}

\newcommand{\qt}{{\mathbb{T}_\theta^d}}

\newcommand{\ri}{{\mathrm{i}}}

\newcommand{\rTr}{{\mathrm{Tr}}}
\newcommand{\real}{{\mathbb{R}}}

\begin{document}
	
	\title{Endpoint Schatten class properties of commutators on quantum tori}

\thanks{{\em 2020 Mathematics Subject Classification:}
Primary 47L20, 35P20.
Secondary 46L51, 47G30, 58B34.}
	
	\thanks{{\em Key words:} Quantum tori, commutators, Schatten spaces, pseudodifferential operators, spectral asymptotics}

	

	


	\author{Xinyu ZHANG}
	
	\address{Institute for Advanced Study in Mathematics, Harbin Institute of Technology, Harbin 150001, China}
	\email{xy.zhang@stu.hit.edu.cn}

	\date{}

	\markboth{X. Zhang}%
	{Endpoint Schatten class properties of commutators}

	\begin{abstract}
		We study endpoint Schatten class properties of commutators
$[I_\theta^\epsilon,M_x]$ on quantum tori of dimension $d\geq2$,
where $I_\theta^\epsilon$
is the Riesz potential of order $\epsilon$
and $M_x$ denotes the left multiplication operator by
an operator $x$ on quantum tori $\mathbb T_\theta^d$.
For $-\frac{d}{2}<\epsilon<1$, $\epsilon\neq0$, and
$x\in L_1(\mathbb T_\theta^d)$, we prove that
$[I_\theta^\epsilon,M_x]\in
\mathcal L_{\frac{d}{1-\epsilon},\infty}$ if and only if
$x\in\dot W_{\frac{d}{1-\epsilon}}^1(\mathbb T_\theta^d)$. We also determine the singular value asymptotics of this commutator for
$x\in\dot W_{\frac{d}{1-\epsilon}}^1(\mathbb T_\theta^d)$.
The necessary Sobolev condition holds throughout
$1-d<\epsilon<1$, $\epsilon\neq0$.
For $\epsilon\leq-\frac{d}{2}$, we characterise
Hilbert--Schmidt membership, with a logarithmically
weighted Fourier condition at $\epsilon=-\frac{d}{2}$.

In order to get the singular value asymptotics of commutators, we prove singular value asymptotics for pseudodifferential
operators of arbitrary negative order on quantum tori
using $C^*$-algebraic principal symbols. This asymptotic result confirms Conjecture~7.1 of McDonald and Ponge
(\emph{Adv. Math.} 412 (2023), 108815).
	\end{abstract}
	\maketitle
	
\section{Introduction}

This paper is a continuation of \cite{SXZ2023}, where the singular
value asymptotics of quantised derivatives on quantum tori $\mathbb T_\theta^d$ were studied.
We are interested in the spectral properties of commutators
\[
 [I_\theta^{\epsilon},M_x]
 \quad\text{on }L_2(\mathbb T_\theta^d).
\]
Here $I_\theta^\epsilon=(-\Delta)^{\frac{\epsilon}{2}}$
denotes the Riesz potential of order $\epsilon$.
Throughout, $d\geq2$, $\theta$ is an arbitrary real antisymmetric
$d\times d$ matrix, and $M_x$ is the operator of left multiplication by an operator $x \in L_1(\mathbb T_\theta^d)$.
We use $-\Delta=\sum_{j=1}^dD_j^2$, with $D_jU^n=n_jU^n$, and focus on the range $1-d<\epsilon <1$.
Negative powers of $-\Delta$ are understood as generalised
inverses.
For $\epsilon\neq0$, we seek to characterise
membership of $[I_\theta^\epsilon,M_x]$ in the critical
weak Schatten class in terms of the regularity of $x$.

The relation between commutators and regularity has a long history.
For the Riesz transforms
$R_j=-\ri\partial_j(-\Delta)^{-\frac{1}{2}}$, $1\leq j\leq d$,
on $\mathbb R^d$, the classical results of Coifman, Rochberg
and Weiss \cite{CRW1976}, Janson \cite{Janson1978} and
Uchiyama \cite{Uchiyama1978} characterise boundedness and
compactness of $[R_j,M_f]$ by
$f\in\mathrm{BMO}(\mathbb R^d)$ and
$f\in\mathrm{VMO}(\mathbb R^d)$, respectively; here $\mathrm{BMO}$ denotes the class of functions of bounded mean oscillation and $\mathrm{VMO}$ denotes the closure of
$C_c^\infty(\mathbb R^d)$ in the $\mathrm{BMO}$ seminorm.
Janson and Wolff \cite{JW1982} characterised its Schatten class
membership for $p>d$ by fractional Sobolev regularity, whereas
membership in $\mathcal L_d$ forces $f$ to be constant.
At the weak endpoint, membership in $\mathcal L_{d,\infty}$ is equivalent
to $f\in\dot W_d^1(\mathbb R^d)$; see
\cite{CST1994,LMSZ2017,FSZ2023}.
For commutators $[(-\Delta)^{\frac{\epsilon}{2}},M_f]$, Janson and Peetre \cite{JP1988}
obtained analogous Schatten characterisations in fractional Sobolev
spaces above the critical exponent; at the critical exponent,
Frank, Sukochev and Zanin
\cite[Theorems~1.1 and~1.2]{FSZ2024} proved that, for $d\geq2$,
$-\frac{d}{2}<\epsilon<1$, $\epsilon\neq0$, and
$f\in L_{1,\mathrm{loc}}(\mathbb R^d)$,
\[
 [(-\Delta)^{\frac{\epsilon}{2}},M_f]
 \in\mathcal L_{\frac{d}{1-\epsilon},\infty}
 \quad\Longleftrightarrow\quad
 f\in\dot W_{\frac{d}{1-\epsilon}}^1(\mathbb R^d).
\]
In the noncommutative setting, McDonald, Sukochev and Xiong
\cite{MSX2019,MSX2020} characterised weak Schatten membership of
quantised derivatives on quantum tori and quantum Euclidean
spaces by first order Sobolev regularity. 
On quantum tori, the Riesz transforms are given by
$R_j=D_jI_\theta^{-1}$, $1\leq j\leq d$.
Theorems~1.1 and~1.4 of \cite{MSX2019} give the following
equivalent characterisation: for $x\in L_2(\qt)$,
\[
 x\in\dot W_d^1(\qt)
 \quad\Longleftrightarrow\quad
 [R_j,M_x]\in\mathcal L_{d,\infty}
 \quad\text{for every }j=1,\ldots,d.
\]

Spectral asymptotics give more precise information than ideal
membership or trace formulas alone.
Frank, Sukochev and Zanin \cite{FSZ2023} established singular
value asymptotics for Euclidean quantised derivatives associated
with $f\in\dot W_d^1(\mathbb R^d)$.
They also obtained singular value asymptotics for commutators $[(-\Delta)^{\frac{\epsilon}{2}},M_f]$ with
$f\in\dot W_{\frac{d}{1-\epsilon}}^1(\mathbb R^d)$ in
\cite{FSZ2024}.
On quantum tori, Sukochev, Xiong and Zanin \cite{SXZ2023}
proved the corresponding asymptotics for the quantised derivatives
of every $x\in\dot W_d^1(\qt)$, strengthening the noncommutative
integration results of \cite{MSX2019}. For $T\in\mathcal L_{p,\infty}$, $p>0$, the implication is
\[
 \lim_{t\to\infty}t^{\frac{1}{p}}\mu(t,T)=c
 \quad\Longrightarrow\quad \varphi(|T|^p)=c^p
\]
for every continuous normalised trace $\varphi$ on
$\mathcal L_{1,\infty}$, in particular for every Dixmier trace. Thus a singular value asymptotic formula yields
a Connes-type integration formula, while such a trace identity does
not in general recover the individual singular value asymptotics.
This is the relation between spectral asymptotics and Connes' trace theorem
\cite{Connes1988} developed in \cite{R2022,SZ2023,Ponge2023}.
Noncommutative Tauberian methods provide a further approach to
such asymptotics and their geometric applications \cite{MSZ2022}.

\medskip

In this paper, we pass from the quantised derivatives studied
in \cite{SXZ2023} to commutators
$[I_\theta^{\epsilon},M_x]$ on quantum tori.
Our main results give quantum torus counterparts of the
Sobolev characterisation and singular value asymptotics
in \cite{FSZ2024}.
However, the compactness properties differ in the lower-order range
$\epsilon\leq-\frac{d}{2}$.
For every $x\in\mathcal S(\qt)$, the commutator
$[I_\theta^{\epsilon},M_x]$ extends to a compact
operator on $L_2(\qt)$ for every $\epsilon<1$.
When $\epsilon<0$, compactness holds even for
$x\in L_\infty(\qt)$, since
$I_\theta^{\epsilon}$ is itself compact.
By contrast, the Euclidean Fourier multiplier
$|\xi|^\epsilon$ is unbounded near the origin.
For $-d<\epsilon\leq-\frac{d}{2}$, the commutator
$[(-\Delta)^{\frac{\epsilon}{2}},M_f]$ does not extend
to a bounded operator on $L_2(\mathbb R^d)$ for any
nonzero $f\in C_c^\infty(\mathbb R^d)$
\cite[Lemma~3.3]{FSZ2024}.
This difference already occurs when $\theta=0$.
After stating our main results, we discuss the range of
$\epsilon$ and explain these differences in more detail.

\subsection*{Main results}
To state our main results, we first fix some notation.
For $0<p<\infty$, we write $L_p(\qt)$ for the associated
noncommutative $L_p$-space, with
$\|x\|_p=\tau(|x|^p)^{\frac1p}$, see
\cite[Sections~1.1--1.2]{XXY2018}.
We denote the smooth algebra and its distribution space by
$\mathcal S(\qt)$ and $\mathcal S'(\qt)$, respectively;
see \cite[Section~2.1]{XXY2018}.
For $1\leq p<\infty$, the homogeneous Sobolev space
$\dot W_p^1(\qt)$ consists of distributions $x$ modulo constants such that
$D_jx\in L_p(\qt)$ for every $j=1,\ldots,d$, and is equipped
with the norm
$$
 \|x\|_{\dot W_p^1}
 =
 \Big(\sum_{j=1}^d\|D_jx\|_p^p\Big)^{\frac1p}.
$$
Further details are given in Section~\ref{subsec-nctori}
and \cite[Section~2.2]{XXY2018}.

	The first main result gives the upper estimate for the commutator $[I_\theta^{\epsilon},M_x]$.
\begin{thm}\label{mta}
		Let $d\geq 2$ and let $-\frac{d}{2}<\epsilon<1$. If $x\in\dot{W}^{1}_{\frac{d}{1-\epsilon}}(\qt)$, then $[I_\theta^{\epsilon},M_x]$ extends to a compact operator in $\mathcal{L}_{\frac{d}{1-\epsilon},\infty}$, with
		$$\Big\|[I_\theta^{\epsilon},M_x]\Big\|_{\frac{d}{1-\epsilon},\infty}\leq c_{d,\epsilon}\|x\|_{\dot{W}^{1}_{\frac{d}{1-\epsilon}}(\qt)}.$$
		\end{thm}

The second main result determines the singular value asymptotics of the same commutator.

\begin{thm}\label{mtb}
		Let $d\geq 2$ and let $-\frac{d}{2}<\epsilon<1$.  If $x\in\dot{W}^{1}_{\frac{d}{1-\epsilon}}(\qt)$, then
		\begin{equation}\label{mtb-eq}
			 \lim_{t\to\infty}t^{\frac{1-\epsilon}{d}}\mu\Big(t,[I_\theta^{\epsilon},M_x]\Big)=  |\epsilon| \, d^{ \frac{\epsilon-1}{d}}    \|  \sum_{1\leq j \leq d }D_jx\otimes s_j \|_{L_{\frac{d}{1-\epsilon}}(L_{\infty}(\mathbb{T}_\theta^d)\bar{\otimes} L_\infty(\mathbb{S}^{d-1}))}.
		\end{equation}
Here $s_j$ is the $j$-th coordinate function on $\mathbb S^{d-1}$,
and the product trace uses the usual unnormalised surface measure.
Moreover, the norm on the right-hand side of \eqref{mtb-eq} is
equivalent to the homogeneous Sobolev norm of $x$. 
	\end{thm}

The third main result gives the lower estimate for the commutator $[I_\theta^{\epsilon},M_x]$ that holds over the extended range $1-d < \epsilon < 1$, $\epsilon \neq 0$.

\begin{thm}\label{mtd}
	Let $d\geq2$, $1-d<\epsilon<1$ with $\epsilon\neq0$. If $x\in L_{1}(\qt)$ and if 	$[I_\theta^{\epsilon},M_x]\in\mathcal{L}_{\frac{d}{1-\epsilon},\infty},$ then $x\in\dot{W}^{1}_{\frac{d}{1-\epsilon}}(\qt)$ and
	$$\|x\|_{\dot{W}^{1}_{\frac{d}{1-\epsilon}}(\qt)}\leq c'_{d,\epsilon}\Big\|[I_\theta^{\epsilon},M_x]\Big\|_{\frac{d}{1-\epsilon},\infty}.$$
\end{thm}

\subsection*{Discussion on the range}

\begin{enumerate}[(i)]
\item
For $d\geq2$, $-\frac{d}{2}<\epsilon<1$ with $\epsilon\neq0$,
Theorems~\ref{mta} and~\ref{mtd} give
\[
 [I_\theta^{\epsilon},M_x]
 \in\mathcal L_{\frac{d}{1-\epsilon},\infty}
 \quad\Longleftrightarrow\quad
 x\in\dot W_{\frac{d}{1-\epsilon}}^1(\qt),
 \qquad x\in L_1(\qt).
\]
Theorem~\ref{mtb} gives the singular value asymptotics
under the same Sobolev assumption.

\item
For $d\geq2$, $\epsilon\leq-\frac{d}{2}$ and $x\in L_1(\qt)$,
Proposition~\ref{prop:HS} gives
\[
 \big\|[I_\theta^{\epsilon},M_x]\big\|_{\mathcal L_2}
 \approx
 \begin{cases}
  \|x-\tau(x)\|_{L_2(\qt)},
      &\epsilon<-\frac{d}{2},\\[1mm]
  \displaystyle
  \Bigg(\sum_{k\ne0}
     \log(1+|k|)|\widehat x(k)|^2\Bigg)^{\frac{1}{2}},
      &\epsilon=-\frac{d}{2}.
 \end{cases}
\]
The constants depend only on $d$ and $\epsilon$.
The commutator is Hilbert--Schmidt if and only if
the right-hand side is finite.

\item
For $d\geq3$ and $1-d<\epsilon\leq-\frac{d}{2}$,
there are elements
$x\in\dot W_{\frac{d}{1-\epsilon}}^1(\mathbb T^d)$
whose commutators have no Hilbert--Schmidt extension.
Thus the conclusions of Theorems~\ref{mta} and~\ref{mtb}
cannot be extended to all elements of this Sobolev space,
even when $\theta=0$; see
Example~\ref{ex:lower-order}.

\item
For $d\geq3$ and $1-d<\epsilon\leq-\frac{d}{2}$,
Theorem~\ref{mtd} still gives
\[
 [I_\theta^{\epsilon},M_x]
 \in\mathcal L_{\frac{d}{1-\epsilon},\infty}
 \quad\Longrightarrow\quad
 x\in\dot W_{\frac{d}{1-\epsilon}}^1(\qt),
 \qquad x\in L_1(\qt).
\]
Since
\[
 1<\frac{d}{1-\epsilon}<2,
 \qquad
 \mathcal L_{\frac{d}{1-\epsilon},\infty}
 \subset\mathcal L_2,
\]
the Hilbert--Schmidt condition in~(ii) is also necessary;
see Corollary~\ref{cor:lower-range-necessary-conditions}.
\end{enumerate}

The critical exponent $\frac{d}{1-\epsilon}$ does not depend on $\theta$.

\subsection*{Pseudodifferential operators and spectral asymptotics}

In the commutative setting, pseudodifferential operators were first explicitly defined by Kohn-Nirenberg \cite{KN1965} and H\"ormander \cite{Hor1965} to connect singular integrals and differential operators, see also \cite{Stein1993}. In the noncommutative setting, the research of pseudodifferential operators started with Connes and Baaj's work \cite{Connes1980, Baaj1988} for $C^*$-dynamical systems; see \cite{Tao2018,HLP2019I,HLP2019II} for recent progress. ``Classical pseudodifferential operators" have been defined in the above mentioned literature as the class of pseudodifferential operators whose symbols admit asymptotic expansions by homogeneous functions on $\mathbb{R}^d$. In addition to the traditional smooth classical calculus, we use the $C^*$-algebraic approach introduced in \cite{SZ2018, MSZ2019}. Following \cite{SZ2018,MSZ2019}, we work with the
$C^*$-algebra $\Pi(C(\mathbb T_\theta^d),C(\mathbb S^{d-1}))$ (see Section \ref{subsec:Cstar-PDO} below for specific definition),
which contains the traditional classical pseudodifferential
operators of order zero and admits continuous principal symbols.
	
	Spectral asymptotics of classical pseudodifferential operators have been known for some time, see for example the work of Birman and Solomyak \cite{BS1970, BS1972}.
In the noncommutative setting developed in this paper, spectral analysis of the leading symbol relies on the $C^*$-algebraic principal symbol framework introduced in \cite{SZ2018, MSZ2019}. In Section~\ref{sec:PDO-Weyl}, we extend
the asymptotic formula for $TI_\theta^{-1}$ in
\cite[Theorem~1.1]{SXZ2023} to
$TJ_\theta^{-m}$ for every $m>0$ and
$T\in\Pi(C(\qt),C(\mathbb S^{d-1}))$.
The case $m=1-\epsilon$ gives the formula needed for
the leading term of $[I_\theta^\epsilon,M_x]$ in
Theorem~\ref{approximation theorem}. We also prove the positive and negative eigenvalue asymptotics.
These prove Conjecture~7.1 of McDonald and Ponge
\cite{MP2023} for selfadjoint classical
pseudodifferential operators of negative order, without an
ellipticity assumption; see
Corollary~\ref{cor:McDonald-Ponge-conjecture}.

\subsection*{Organisation}

Section~\ref{sec-pre} recalls the basic notation for quantum
tori and collects preliminaries on operator ideals,
double operator integrals, $C^*$-algebraic pseudodifferential
operators, and commutator estimates.
Section~\ref{sec:boundedness-HS} defines the distributional
commutator and establishes its preliminary boundedness properties
and Hilbert--Schmidt criteria. Section~4 proves
Theorem~\ref{mta}. Section~\ref{sec:PDO-Weyl} develops the
spectral asymptotics used in the commutator argument and treats
the classical pseudodifferential application.
Section~6 identifies the leading term of the commutator modulo
a smaller weak Schatten ideal. Section~7 contains the proofs
of Theorems~\ref{mtb} and \ref{mtd}.

	\section{Preliminaries}\label{sec-pre}
	
\subsection{Quantum tori}\label{subsec-nctori}

We recall some basic facts about quantum tori, following
\cite[Section~1.2]{XXY2018}.  We fix an integer $d\geq2$ and $\theta = \{\theta_{j,k}\}_{j,k = 1}^d$, a $d\times d$ antisymmetric real matrix. The $C^*$-algebra of continuous functions on the noncommutative torus, denoted $C(\qt)$, is the universal $C^*$-algebra on $d$ unitary generators $U_1,\cdots,U_d$
	which satisfy:
	\begin{equation*}
		U_jU_k = e^{{\rm i}\theta_{j,k}}U_kU_j,\quad 1\leq j,k\leq d.
	\end{equation*}
	Given $n = (n_1,\cdots,n_d) \in \mathbb{Z}^d$, we adopt the shorthand notation:
	\begin{equation*}
		U^n := U_1^{n_1}U_2^{n_2}\cdots U_d^{n_d}.
	\end{equation*}
	A polynomial in $U$ is a finite sum
	$$x   =   \sum_{m  \in \mathbb{Z}^d}    \alpha_m  U^m    \;\;\mbox{with } \;\; \alpha _m  \in \mathbb{C} . $$
	The involutive algebra of all such polynomials is dense in $C(\qt)$. For any polynomial $x$ as above, we define $\tau(x) = \alpha _0$. Then $\tau $ extends to a faithful tracial state on $C(\qt)$, still denoted as $\tau $. We now define the GNS Hilbert space $L_2(C(\qt),\tau)$, which we denote $L_2(\qt)$, and we represent $C(\qt)$ as an algebra of bounded
	operators on $L_2(\qt)$, where $x \in C(\qt)$ acts on $\xi \in L_2(\qt)$ by left multiplication. Taking the weak operator topology closure $C(\qt)''$ in $\Bc(L_2(\qt))$ yields a von Neumann algebra, which we denote $L_\infty(\qt)$.
	
	The $L_p$-spaces for $p \in [1,\infty)$ on $\qt$ are then defined as the noncommutative $L_p$-spaces \cite{LSZ2012} (see also \cite[Section~6.2]{DPS2023}) on $(L_\infty(\qt),\tau)$.
	For $x \in L_1(\qt)$ and $n\in \mathbb{Z}^d$, we define:
	\begin{equation*}
		\widehat{x}(n) = \tau(x(U^{n})^{*}).
	\end{equation*}
	In particular, this definition applies to $x\in L_2(\qt)\subset L_1(\qt).$
	By the definition of $\tau$, we see that $\tau(U^n) = \delta_{n,0}$, and then standard Hilbert space arguments show that any $x \in L_2(\qt)$ can be written as an $L_2$-convergent
	series:
	\begin{equation*}
		x = \sum_{n \in \mathbb{Z}^d} \widehat{x}(n)U^n,
	\end{equation*}
	with
	\begin{equation}\label{Plancherel-qt}
		\|x\|_2^2 = \sum_{n\in \mathbb{Z}^d} | \widehat x(n) |^2 .
	\end{equation}

	The space $\mathcal{S}(\qt)$ is defined to be the subset of $x\in C(\qt)$ such that the sequence of Fourier coefficients $\{\widehat{x}(n)\}_{n\in \mathbb{Z}^d}$
	has rapid decay (i.e., the sequence $\{|\widehat{x}(n)|\}_{n\in \mathbb{Z}^d}$ is eventually dominated by the reciprocal of any polynomial). We may consider $\mathcal{S}(\qt)$
	as the analogue of the space of $C^\infty$ functions on $\mathbb{T}^d$. There is also a canonical Fr\'echet
	topology on $\mathcal{S}(\qt)$, and the space $\mathcal{S}'(\qt)$, called the space of distributions on $\qt$, is defined to be the topological dual of $\mathcal{S}(\qt)$. See \cite[Section~2.1]{XXY2018} for more details.

	The partial differentiation operators $D_j,$ $j= 1,\cdots,d$ on quantum tori are defined by:
	\begin{equation*}
		D_j (U^n) = n_jU^n,\quad n = (n_1,\cdots,n_d) \in \mathbb{Z}^d.
	\end{equation*}
	Every partial derivation $D_j$ can be viewed as an (unbounded) self-adjoint   operator on $L_2(\qt).$ Let $\Delta=-\sum_{j=1}^dD_j^2$ be the Laplacian. The operator $-\Delta$ is a positive and self-adjoint operator on $L_2(\qt)$ with spectrum equal to $\{|n|^2:\ n \in \mathbb{Z}^d\}.$
	
	Given $n =(n_1,\cdots,n_d)\in \mathbb{N}_0^d$ ($\mathbb{N}_0$ denoting the set of nonnegative integers), the associated partial derivation $D^n$ is defined to be
	$D_1^{n_1}\cdots D_d^{n_d}$. The order of $D^n$ is   $|n|_1=n_1+\cdots+ n_d$.
	By duality, the derivations transfer to  $\mathcal{S}'(\qt)$  as well.
		The Sobolev space of order $k$  on $\mathbb{T}_{\theta}^d$ is defined to be
	$$W_p^k(\mathbb{T}_{\theta}^d)= \big\{ x\in\mathcal{S}'(\qt) :  D^n x \in L_p(\mathbb{T}_{\theta}^d) \textrm{ for each }n\in \mathbb {N}_0^d \textrm{ with } |n |_1\leq k \big\},$$
	equipped with the  norm
	$$\|x\|_{W_p^k}=\Big(\sum_{0\leq |n |_1\leq k}\|D^n x\|_{p}^p\Big)^{\frac{1}{p}}.$$
	In this paper, we will mainly use the ``homogeneous" Sobolev space $\dot{W}^1_p (\qt)$, the norm of which may be described in the following equivalent forms:
	\begin{equation}\label{Sob-equi-norm}
		\|x\|_{\dot{W}^1_p}=\Big(\sum_{j=1}^d \|D_j x\|_p^p \Big)^{\frac{1}{p}}.
	\end{equation}

For $1\leq p \leq\infty$, the
Poincar\'e inequality \cite[Theorem~2.12]{XXY2018} gives
\begin{equation}\label{eq:distribution-Poincare}
 x-\widehat x(0)\in L_p(\qt),\qquad
 \|x-\widehat x(0)\|_p
 \leq C_{d,p}\|x\|_{\dot W_p^1},
 \qquad x\in\dot W_p^1(\qt).
\end{equation}

	Fourier multipliers for quantum tori were studied in detail in \cite[Chapter 7]{XXY2018}. Let $g$ be a bounded scalar function on $\mathbb{Z}^d$. For $x \in L_2(\qt)$, the Fourier multiplier $T_g$ with symbol $g$ is defined on $x$ by:
	\begin{equation}\label{def-Fourier-Z}
		T_gx = \sum_{n \in \mathbb{Z}^d} g(n)\widehat{x}(n)U^n.
	\end{equation}
	By virtue of the Plancherel identity \eqref{Plancherel-qt}, $T_g$ indeed defines a bounded linear operator on $L_2(\qt)$ and the above series converges in the $L_2$-sense.
	
The operators $D_j$ and $\Delta$ are Fourier multipliers with
symbols $n_j$ and $-|n|^2$, respectively. Let $P_0$ be the
orthogonal projection of $L_2(\qt)$ onto $\mathbb C1$.
For $\epsilon\in\mathbb R$, we denote the Riesz and Bessel
potentials, respectively, by
\[
 I_\theta^\epsilon=(-\Delta)^{\frac{\epsilon}{2}},
 \qquad
 J_\theta^\epsilon=(1-\Delta)^{\frac{\epsilon}{2}}.
\]
As in \cite[p.~16]{SXZ2023}, negative powers of $-\Delta$
are defined by spectral calculus on
$(1-P_0)L_2(\qt)$ and extended by zero on $\mathbb C1$.
At order zero, we take $I_\theta^0=1$. Thus
\begin{equation}\label{eq:a-epsilon}
 I_\theta^\epsilon U^n=a_\epsilon(n)U^n,
 \qquad
 a_\epsilon(n)=
 \begin{cases}
  |n|^\epsilon,&n\neq0,\\
  0,&n=0,\ \epsilon\neq0,\\
  1,&n=0,\ \epsilon=0.
 \end{cases}
\end{equation}
Both $I_\theta^\epsilon$ and $J_\theta^\epsilon$ act continuously
on $\mathcal S(\qt)$ and, by duality, on $\mathcal S'(\qt)$.
Negative powers of the absolute value of the Dirac operator
introduced below are understood in the same way.

		The tuple $(D_1,\cdots,D_d)$ is a commuting self-adjoint one with the common core $W^1_2(\qt).$ We denote this tuple by $\nabla.$
		From the perspective of continuous functional calculus, we may also write the Fourier multiplier as
	\begin{equation*}
		T_g = g(D_1,\cdots,D_d).
	\end{equation*}

	The Dirac operator $\D$ is defined in terms of $\gamma$ matrices in direct analogy to commutative tori. Define $\mathbf{d} = 2^{\lfloor \frac{d}{2}\rfloor}$ and select $\mathbf{d}\times\mathbf{d}$ complex self-adjoint matrices $\{\gamma_1,\cdots,\gamma_d\}$ satisfying $\gamma_j\gamma_k +\gamma_k\gamma_j = 2\delta_{j,k}1_{\mathbf{d}\times \mathbf{d}}$, and define:
	\begin{equation*}
		\D = \sum_{j=1}^d \gamma_j\otimes D_j
	\end{equation*}
	as an unbounded, densely defined linear operator on the Hilbert space $\mathbb{C}^{\mathbf{d}}\otimes L_2(\qt)$. Since all $D_j$'s
	are self-adjoint, $\D$ is also self-adjoint. $\D$ is the square root of $-\Delta$ in the sense that $\D^2 = -1_{\mathbf{d}\times \mathbf{d}}\otimes \Delta$. Even though it is not necessary to introduce the Dirac operator $\D$ in this paper, the use of this notion will give more convenience when we represent the commutators by double operator integrals.
	
	Given $x\in L_\infty(\qt)$, denote by $M_x: y \mapsto xy$ the operator of left multiplication on $L_2(\qt)$. The operator $1\otimes M_x$ is a bounded linear operator on $\mathbb{C}^{\mathbf{d}} \otimes  L_2(\qt) $, where $1$ denotes the identity operator on $\mathbb{C}^{\mathbf{d}}$.
	
	\subsection{Operators, Ideals and traces}\label{operator notation subsection}
	The following material concerning operator ideals and traces is standard. For more details we refer the reader to \cite{DPS2023, LSZ2012, LSZ2021, Simon1979}. Let $H$ be a complex separable Hilbert space, and let $\Bc(H)$ denote the set of bounded operators on $H$, and let $\mathcal{K}(H)$ denote the ideal of compact operators on $H$. Given $T\in \mathcal{K}(H)$, its sequence of singular values is denoted by $\mu(T) = \{\mu(k,T)\}_{k=0}^\infty.$ For real $t\geq0$, we set
$\mu(t,T)=\mu(\lfloor t\rfloor,T)$.
We also set $\mathcal L_\infty=\mathcal B(H)$,
with the operator norm $\|\cdot\|_\infty$.
	
	The following basic properties of singular values will be frequently used in the sequel:
	\begin{equation}\label{singular-value-sum}
		\mu(t+s , A+B)  \leq  \mu (t, A) + \mu (s,  B),
	\end{equation}
	and
	\begin{equation}\label{singular-value-square}
		\mu(t , A)  =  \mu (t, A^*)  =   \mu (t,  |A|^p)^{\frac{1}{p}},
	\end{equation}
	for any $0<p<\infty$.
	See e.g. \cite[Section~2.3]{LSZ2012} or \cite[Section~1.5]{Xu2007}. See also \cite[Section~3.2]{DPS2023}.

	Let $p \in (0,\infty).$ The Schatten class $\mathcal{L}_p$ is the set of operators $T$ in $\mathcal{K}(H)$ such that $\mu(T)$ is $p$-summable, i.e. in the sequence space $\ell_p$. If $1\leq p <\infty$ then the $\mathcal{L}_p$
	norm is defined as:
	\begin{equation*}
		\|T\|_p := \|\mu(T)\|_{\ell_p} = \big(\sum_{k=0}^\infty \mu(k,T)^p\big)^{\frac{1}{p}}.
	\end{equation*}
	With this norm $\mathcal{L}_p$ for $1\leq p <\infty$ is a Banach space, and a Banach ideal of $\Bc(H)$.
	
	Analogously, for $0<p< \infty$, the weak Schatten class $\mathcal{L}_{p,\infty}$ is the set of operators $T\in \mathcal{K}(H)$ such that $\mu(T)$ is in the weak $L_p$-space $\ell_{p,\infty}$, with quasi-norm:
	\begin{equation*}
		\|T\|_{p,\infty} = \sup_{k\geq 0} (k+1)^{\frac{1}{p}}\mu(k,T) < \infty.
	\end{equation*}
	As with the $\mathcal{L}_p$ spaces, $\mathcal{L}_{p,\infty}$ is an ideal of $\Bc(H)$. We also have the following form of H\"older's inequality:
\begin{equation}\label{weak-type-holder}
 \|TS\|_{r,\infty}
 \leq
 \frac{(p+q)^{\frac1p+\frac1q}}
             {q^{\frac1p}p^{\frac1q}}
 \|T\|_{p,\infty}\|S\|_{q,\infty},
\end{equation}
where $p,q,r>0$ and
$\frac1r=\frac1p+\frac1q$.
This holds for $T\in\mathcal L_{p,\infty}$ and
$S\in\mathcal L_{q,\infty}$.
The constant is optimal; see
\cite[Theorem~4]{SZHolder2021}.

For $q>0$, $(\mathcal L_{q,\infty})_0$ is the closure of the
finite-rank operators, equivalently the set of $R$ satisfying
$t^{\frac{1}{q}}\mu(t,R)\to0$. For $q,r>0$, finite-rank approximation and
\eqref{weak-type-holder} show that
\begin{equation}\label{eq:small-ideal-product}
 (\mathcal L_{q,\infty})_0\mathcal L_{r,\infty}
 +\mathcal L_{r,\infty}(\mathcal L_{q,\infty})_0
 \subset(\mathcal L_{u,\infty})_0,\qquad
 \frac{1}{u}=\frac{1}{q}+\frac{1}{r}.
\end{equation}

	%
	%
	
	\subsection{Double operator integrals}
	
	Let $H$ be a (complex) separable Hilbert space. Let $D_0$ and $D_1$ be self-adjoint (potentially unbounded) operators on $H$, and $E^0$ and $E^1$ be the associated spectral measures. For all $x, y \in \mathcal{L}_2(H)$, the measure $(\lambda, \mu)  \mapsto \rTr(x\, dE^0(\lambda) \, y \, dE^1(\mu)  )$ is a countably additive complex valued measure on $\real^2$. We say that $\phi \in L_{\infty}(\real^2)$ is $E^0 \otimes E^1$ integrable if there exists an operator $T_\phi ^{D_0, D_1} \in \mathcal{B}  (\mathcal{L}_2(H))$ such that for all $x, y \in \mathcal{L}_2(H)$,
	$$\rTr (x\,T_\phi ^{D_0, D_1} y  ) =\int _{\real^2}    \phi(\lambda, \mu )   \rTr(x\, dE^0(\lambda) \, y \, dE^1(\mu)  ). $$
	The operator $T_\phi ^{D_0, D_1} $ is called the transformer. For $A\in \mathcal{L}_2(H) $, we define
	\begin{equation}\label{doi-def}
		T_\phi ^{D_0, D_1}(A)=\int _{\real^2}    \phi(\lambda, \mu )   dE^0(\lambda) \, A \, dE^1(\mu)  . \end{equation}
	This is called a double operator integral.

\subsection{A \texorpdfstring{$C^*$}{C*}-algebra of
pseudodifferential operators}
\label{subsec:Cstar-PDO}

We recall the $C^*$-algebraic principal symbol construction of
\cite{SZ2018,MSZ2019}, as in \cite{SXZ2023}. Let
$\Pi=\Pi(C(\qt),C(\mathbb S^{d-1}))$ be the norm-closed
$*$-algebra generated by the left multiplication operators
$\pi_1(x)=M_x$, $x\in C(\qt)$, and the Fourier multipliers
defined by
\[
 \pi_2(g)U^n=g\!\left(\frac{n}{|n|}\right)U^n,
 \quad n\neq0,\qquad \pi_2(g)1=0,
 \quad g\in C(\mathbb S^{d-1}).
\]
This defines a faithful $*$-representation.
By \cite[Corollary~3.7]{SXZ2023}, the commutators
$[\pi_1(x),\pi_2(g)]$ are compact. The principal
symbol construction
\cite[Theorem~3.3 and Lemma~4.5]{MSZ2019} gives a surjective
$*$-homomorphism
\[
 \sym:\Pi\longrightarrow C(\qt)\otimes_{\min}C(\mathbb S^{d-1}),
 \qquad \sym(M_x\pi_2(g))=x\otimes g,
\]
with kernel $\Pi\cap\mathcal K(L_2(\qt))$.
Thus $\sym(T)=\sym(S)$ if and only if $T-S$ is compact.
The symbol algebra is $C(\mathbb S^{d-1};C(\qt))$. Note that the $C^*$-algebra $\Pi$ plays the role of the algebra of classical pseudodifferential operators of order $0$, which contains continuous but non-smooth elements. 
Continuous functional calculus in $\Pi$ is compatible with $\sym$;
in particular, $\sym(A_\pm)=\sym(A)_\pm$ for $A=A^*\in\Pi$.


	\subsection{Commutator estimates}\label{sec-commutator}

	In this subsection, we collect some commutator estimates for Bessel and Riesz potentials $J_\theta^{\alpha}$ and $I_\theta^{\alpha}$ on quantum tori, that will be useful in the proofs of the main theorems. Firstly, let us recall the following estimate from \cite{SXZ2023}.
	
	\begin{thm}\label{commutator-PSD}
		Let $\alpha,  \beta \in \mathbb{R}$ and let $x\in \mathcal{S}(\qt).$ If $\alpha  <\beta +1,$ then
		$$[\pi_1 (x),I_\theta^{ \alpha } ] I_\theta^{- \beta } \in \mathcal{L}_{\frac{d}{\beta-\alpha +1},\infty}.$$
	\end{thm}

	It will be useful for us to extend the above estimates for $[T,I_\theta^{ \alpha } ] $, where $T\in \Pi(C(\mathbb{T}_\theta^d),C(\mathbb{S}^{d-1}))$. In order to deal with such non-smooth elements, we need to approximate them with smooth ones.
	
\begin{lem}\label{commutator separable part abstract lemma}
Let $r>0$, let $\{A_i\}_{i\in\mathbb{I}}\subset\mathcal{L}_{\infty}$,
and let $B=B^*\in\mathcal{L}_{r,\infty}$.
If $[A_i,B]\in(\mathcal{L}_{r,\infty})_0$ for every
$i\in\mathbb{I}$, then $[T,B]\in(\mathcal{L}_{r,\infty})_0$
for every $T\in C^*(\{A_i\}_{i\in\mathbb{I}})$.
\end{lem}

\begin{proof}
Since $B=B^*$, we have
\[
[A_i^*,B]=-[A_i,B]^*\in(\mathcal{L}_{r,\infty})_0
\]
for every $i\in\mathbb{I}$.
Let $\mathbb{J}$ be the collection of all finite nonempty tuples
with entries in $\{A_i,A_i^*:i\in\mathbb{I}\}$.
For every $j=(S_1,\ldots,S_n)\in\mathbb{J}$, set
$C_j=S_1\cdots S_n$.
By the Leibniz rule, we have
$[C_j,B]\in(\mathcal{L}_{r,\infty})_0$
for every $j\in\mathbb{J}$.

Let $X=\operatorname{span}\{C_j:j\in\mathbb J\}$.
Hence $[T,B]\in(\mathcal L_{r,\infty})_0$ for every $T\in X$.

Let $T\in C^*(\{A_i\}_{i\in\mathbb{I}})$.
Choose a sequence $\{T_k\}_{k\geq0}\subset X$ such that
$T_k\to T$ in the uniform norm. Then
\[
\|[T_k-T,B]\|_{r,\infty}
\leq c_r\|T_k-T\|_\infty\|B\|_{r,\infty}
\longrightarrow0.
\]
Since $[T_k,B]\in(\mathcal{L}_{r,\infty})_0$ and this subspace
is closed in $\mathcal{L}_{r,\infty}$, we conclude that
$[T,B]\in(\mathcal{L}_{r,\infty})_0$.
This completes the proof.
\end{proof}

	\begin{cor}\label{commutator-PSD-dp}
		Let $d\geq 2$ and $0<\alpha<\infty.$ For every $T\in \Pi(C(\mathbb{T}_\theta^d),C(\mathbb{S}^{d-1})),$ we have
		$$[T,I_\theta^{-\alpha}]\in(\mathcal{L}_{\frac{d}{\alpha},\infty})_0.$$
	\end{cor}

	\begin{proof} If $x\in C(\qt),$ then Theorem \ref{commutator-PSD} and Lemma \ref{commutator separable part abstract lemma} yield
		$$[\pi_1(x),I_{\theta}^{-\alpha}]\in(\mathcal{L}_{\frac{d}{\alpha},\infty})_0.$$
		It is immediate that, for every $g\in C(\mathbb{S}^{d-1}),$
		$$[\pi_2(g),I_{\theta}^{-\alpha}]=0.$$
		The assertion follows from Lemma \ref{commutator separable part abstract lemma}.
	\end{proof}

	\section{Boundedness and Hilbert--Schmidt properties of the commutator}
\label{sec:boundedness-HS}

We define the commutator for $x\in L_1(\qt)$ and establish the
preliminary mapping properties needed in the proof of
Theorem~\ref{mta}. We then give a Hilbert--Schmidt criterion for
$\epsilon\leq-\frac{d}{2}$ and compare this range with the
Euclidean case.

\subsection{Definition and a priori boundedness}
\label{subsec:commutator-definition}

We identify $z\in L_1(\qt)$ with the distribution
$\eta\mapsto\tau(z\eta)$ on $\mathcal S(\qt)$.
We write $\langle\eta,u\rangle=u(\eta^*)$ for
$\eta\in\mathcal S(\qt)$ and $u\in\mathcal S'(\qt)$.
For $u\in L_2(\qt)$, this is the inner product
$\tau(\eta^*u)$.

\begin{lem}\label{lem-definition}
Let $\epsilon\in\mathbb R$ and $x\in L_1(\qt)$. For every
$\xi\in\mathcal S(\qt)$, the expressions
$I_\theta^\epsilon M_x\xi$ and $M_xI_\theta^\epsilon\xi$
define elements of $\mathcal S'(\qt)$, uniquely determined by
\begin{equation}\label{eq:dis}
 \langle\eta,I_\theta^\epsilon M_x\xi\rangle
 =\tau\bigl((I_\theta^\epsilon\eta)^*x\xi\bigr),
 \qquad
 \langle\eta,M_xI_\theta^\epsilon\xi\rangle
 =\tau\bigl(\eta^*xI_\theta^\epsilon\xi\bigr),
 \qquad \eta\in\mathcal S(\qt).
\end{equation}
\end{lem}

\begin{proof}
The symbol $a_\epsilon$ defined in \eqref{eq:a-epsilon} grows at most polynomially, so
$I_\theta^\epsilon$ acts continuously on $\mathcal S(\qt)$.
For example, if $N>d+\max\{\epsilon,0\}$, then
\[
 \|I_\theta^\epsilon\eta\|_\infty
 \leq\sum_{n\in\mathbb Z^d}|a_\epsilon(n)|\,|\widehat\eta(n)|
 \leq C_N\sup_{n\in\mathbb Z^d}
              (1+|n|)^N|\widehat\eta(n)|.
\]
The functions
\[
 \eta\longmapsto\tau(x\xi I_\theta^\epsilon\eta),
 \qquad
 \eta\longmapsto\tau(xI_\theta^\epsilon\xi\,\eta)
\]
are continuous linear functionals on $\mathcal S(\qt)$, since
\[
 |\tau(x\xi I_\theta^\epsilon\eta)|
 \leq\|x\|_1\|\xi\|_\infty\|I_\theta^\epsilon\eta\|_\infty,
 \qquad
 |\tau(xI_\theta^\epsilon\xi\,\eta)|
 \leq\|x\|_1\|I_\theta^\epsilon\xi\|_\infty\|\eta\|_\infty.
\]
These functionals define the two distributions. Since $a_\epsilon$
is real and even,
$I_\theta^\epsilon(\eta^*)=(I_\theta^\epsilon\eta)^*$.
Replacing $\eta$ by $\eta^*$ and using the trace property gives
the identities in \eqref{eq:dis}. They determine both distributions uniquely.
\end{proof}

\begin{defi}\label{def-commutator}
For $\epsilon\in\mathbb R$ and $x\in L_1(\qt)$, define
\[
 [I_\theta^\epsilon,M_x]\xi
 =I_\theta^\epsilon M_x\xi-M_xI_\theta^\epsilon\xi,
 \qquad \xi\in\mathcal S(\qt),
\]
using the distributions in Lemma~\ref{lem-definition}.
This is a continuous linear map from $\mathcal S(\qt)$ to
$\mathcal S'(\qt)$. Membership of this commutator in an operator
ideal means that it has a bounded extension on $L_2(\qt)$
belonging to that ideal.
\end{defi}

The ordering of the monomials gives
\begin{equation}\label{eq:cocycle}
 U^kU^n=\sigma(k,n)U^{k+n},
 \qquad
 \sigma(k,n)=
 \exp\!\left(-\ri\sum_{1\leq r<s\leq d}
                    \theta_{r,s}k_s n_r\right).
\end{equation}
In particular, $|\sigma(k,n)|=1$.

\begin{lem}\label{lem:matrix}
Let $\epsilon<1$ and $x\in L_1(\qt)$. For all
$m,n\in\mathbb Z^d$, the distributional Fourier matrix of the
commutator is
\begin{equation}\label{eq:matrix}
 \big\langle U^m,[I_\theta^\epsilon,M_x]U^n\big\rangle
 =\sigma(m-n,n)\widehat x(m-n)
  \bigl(a_\epsilon(m)-a_\epsilon(n)\bigr).
\end{equation}
Here $a_\epsilon$ is defined in \eqref{eq:a-epsilon}.
\end{lem}

\begin{proof}
For $x\in L_1(\qt)$, \eqref{eq:cocycle} gives
\[
 U^n(U^m)^*=\sigma(m-n,n)(U^{m-n})^*,
\]
and therefore
\[
 \widehat{xU^n}(m)
 =\tau\bigl(xU^n(U^m)^*\bigr)
 =\sigma(m-n,n)\widehat x(m-n).
\]
The $m$th Fourier coefficient of $I_\theta^\epsilon(xU^n)$ is
$a_\epsilon(m)\widehat{xU^n}(m)$, and that of
$xI_\theta^\epsilon U^n$ is
$a_\epsilon(n)\widehat{xU^n}(m)$.
Their difference proves \eqref{eq:matrix}.
\end{proof}

\begin{prop}\label{prop:negative-order-boundedness}
Let $-\frac{d}{2}<\epsilon<0$ and
$x\in\dot W^1_{\frac{d}{1-\epsilon}}(\qt)$. Then
$[I_\theta^\epsilon,M_x]$ extends to a bounded operator on
$L_2(\qt)$.
\end{prop}

\begin{proof}
Subtracting a constant, we may assume that $\tau(x)=\widehat x(0)=0$.
Put
\[
 q=-\frac{d}{\epsilon}>2,
 \qquad
 \frac{1}{r}=\frac{1}{2}+\frac{1}{q}.
\]
By \eqref{eq:distribution-Poincare} and
\cite[Theorem~6.6(i)]{XXY2018}, we have
\[
 \|x\|_q\lesssim_{d,\epsilon}\|x\|_{\dot W_{\frac{d}{1-\epsilon}}^1}.
\]
Also, $r>1$ since $q>2$.

By \cite[Theorems~2.8 and~6.6(i)]{XXY2018}, for $\alpha>0$,
\[
 \|I_\theta^{-\alpha}u\|_b
 \lesssim_{d,\alpha,a}\|u\|_a,
 \qquad
 1<a<b<\infty,\quad
 \frac1b=\frac1a-\frac\alpha d.
\]
For $\eta\in L_2(\qt)$, take $\alpha=-\epsilon$ and $(a,b)=(r,2)$.
H\"older's inequality then gives
\[
 \|I_\theta^\epsilon(x\eta)\|_2
 \lesssim_{d,\epsilon}\|x\eta\|_r
 \leq\|x\|_q\|\eta\|_2.
\]
Taking $(a,b)=(2,\frac{2d}{d+2\epsilon})$
in the preceding estimate gives
\[
 \|I_\theta^\epsilon\eta\|_{\frac{2d}{d+2\epsilon}}
 \lesssim_{d,\epsilon}\|\eta\|_2.
\]
Since $\frac{1}{q}+\frac{d+2\epsilon}{2d}=\frac{1}{2}$,
\[
 \|xI_\theta^\epsilon\eta\|_2
 \leq\|x\|_q
      \|I_\theta^\epsilon\eta\|_{\frac{2d}{d+2\epsilon}}
 \lesssim_{d,\epsilon}\|x\|_q\|\eta\|_2.
\]
Both terms therefore define bounded operators on $L_2(\qt)$,
and their difference extends the distributional commutator.
\end{proof}

\begin{prop}\label{prop:positive-order-domain}
Let $0<\epsilon<1$ and
$x\in\dot W^1_{\frac{d}{1-\epsilon}}(\qt)$. Then
$[I_\theta^\epsilon,M_x]\xi\in L_2(\qt)$ for every
$\xi\in\mathcal S(\qt)$.
\end{prop}

\begin{proof}
We may again assume that $\tau(x)=0$. We have
$\frac{d}{1-\epsilon}>2$. The Poincar\'e inequality
\eqref{eq:distribution-Poincare} gives
$x\in L_{\frac{d}{1-\epsilon}}(\qt)$. For $\xi\in\mathcal S(\qt)$,
\[
 D_j(x\xi)=(D_jx)\xi+xD_j\xi,
 \qquad 1\leq j\leq d.
\]
Since $\xi$ and $D_j\xi$ are bounded and
$L_{\frac{d}{1-\epsilon}}(\qt)\subset L_2(\qt)$, both $x\xi$ and its first derivatives
belong to $L_2(\qt)$. More explicitly,
\[
 \|D_j(x\xi)\|_2
 \leq\|D_jx\|_2\|\xi\|_\infty
       +\|x\|_2\|D_j\xi\|_\infty<\infty.
\]
Thus $y=x\xi\in W_2^1(\qt)$. By \eqref{eq:a-epsilon} and
Plancherel's identity \eqref{Plancherel-qt},
\[
 \|I_\theta^\epsilon y\|_2^2
 =\sum_{n\neq0}|n|^{2\epsilon}|\widehat y(n)|^2
 \leq\sum_{n\neq0}|n|^2|\widehat y(n)|^2
 =\sum_{j=1}^d\|D_jy\|_2^2<\infty.
\]
Also, $I_\theta^\epsilon\xi\in\mathcal S(\qt)$, so
$xI_\theta^\epsilon\xi\in L_2(\qt)$. Subtracting the two terms
proves the assertion.
\end{proof}

\subsection{A Hilbert--Schmidt criterion}
\label{subsec:HS-criterion}

We now consider $\epsilon\leq-\frac{d}{2}$. The Fourier matrix
gives the following criterion, with constants independent of
$\theta$.

\begin{prop}\label{prop:HS}
Let $x\in L_1(\qt)$. If $\epsilon<-\frac{d}{2}$, then
\[
 [I_\theta^\epsilon,M_x]\in\mathcal L_2
 \quad\Longleftrightarrow\quad x-\tau(x)\in L_2(\qt).
\]
When these conditions hold,
\[
 \|[I_\theta^\epsilon,M_x]\|_{\mathcal L_2}
 \approx
 \|x-\tau(x)\|_{L_2(\qt)}.
\]
If $\epsilon=-\frac{d}{2}$, then
\[
 [I_\theta^\epsilon,M_x]\in\mathcal L_2
 \quad\Longleftrightarrow\quad
 \sum_{k\ne0}\log(1+|k|)|\widehat x(k)|^2<\infty.
\]
When these conditions hold,
\[
 \|[I_\theta^\epsilon,M_x]\|_{\mathcal L_2}
 \approx
 \left(
   \sum_{k\ne0}
   \log(1+|k|)|\widehat x(k)|^2
 \right)^{\frac{1}{2}}.
\]
The constants depend only on $d$ and $\epsilon$.
\end{prop}

\begin{proof}
For $k\in\mathbb Z^d$, put
\[
 \rho_\epsilon(k)
 =\sum_{n\in\mathbb Z^d}
   |a_\epsilon(n+k)-a_\epsilon(n)|^2.
\]
In particular, $\rho_\epsilon(0)=0$.
Here $a_\epsilon$ is defined in \eqref{eq:a-epsilon}.
For fixed $k\neq0$ and $|n|>2|k|$, the mean value theorem gives
\[
 |a_\epsilon(n+k)-a_\epsilon(n)|
 \lesssim_\epsilon |k|\,|n|^{\epsilon-1}.
\]
Hence the defining series converges.
By Lemma~\ref{lem:matrix} and $|\sigma(m-n,n)|=1$,
\begin{align*}
 \sum_{m,n\in\mathbb Z^d}
 \big|\langle U^m,[I_\theta^\epsilon,M_x]U^n\rangle\big|^2
 &=
 \sum_{m,n\in\mathbb Z^d}
 |\widehat x(m-n)|^2
 |a_\epsilon(m)-a_\epsilon(n)|^2\\
&=\sum_{k\ne0}
   |\widehat x(k)|^2\rho_\epsilon(k).
\end{align*}
The commutator has a Hilbert--Schmidt extension if and only if
this sum is finite. If the sum is finite, the matrix defines
such an extension. Every extension has the Fourier matrix in
\eqref{eq:matrix}. In that case,
\[
 \|[I_\theta^\epsilon,M_x]\|_{\mathcal L_2}^2
 =\sum_{k\ne0}
   |\widehat x(k)|^2\rho_\epsilon(k).
\]

Assume first that $\epsilon<-\frac{d}{2}$. Then
$a_\epsilon\in\ell_2(\mathbb Z^d)$ and
\[
 \rho_\epsilon(k)
 =\|a_\epsilon(\cdot+k)-a_\epsilon\|_{\ell_2}^2
 \leq4\|a_\epsilon\|_{\ell_2}^2.
\]
Since translates of an $\ell_2$ sequence converge weakly to zero
as $|k|\to\infty$,
\[
 \rho_\epsilon(k)\longrightarrow
 2\|a_\epsilon\|_{\ell_2}^2.
\]
For every $k\neq0$, the term $n=0$ gives
\[
 \rho_\epsilon(k)\geq|a_\epsilon(k)|^2>0.
\]
Together with the preceding limit, this yields
\begin{equation}\label{eq:rho-strict}
 \rho_\epsilon(k)\approx_{d,\epsilon}1,
 \qquad k\neq0.
\end{equation}

Now let $\epsilon=-\frac{d}{2}$ and $K=|k|\geq8$.
If $1\leq|n|\leq\frac{K}{4}$, then
\[
 |n+k|\geq\frac{3K}{4}\geq3|n|,
\]
so
\[
 \big||n+k|^{-\frac{d}{2}}-|n|^{-\frac{d}{2}}\big|^2
 \gtrsim_d |n|^{-d}.
\]
Since
\[
 \sum_{1\leq|n|\leq R}|n|^{-d}
 \approx_d\log(2+R),
 \qquad R\geq2,
\]
we obtain
\begin{equation}\label{eq:rho-endpoint-lower}
 \rho_{-\frac{d}{2}}(k)\gtrsim_d\log(2+K).
\end{equation}

For the upper bound, split the sum into $|n|\leq2K$ and
$|n|>2K$. On the first region,
\[
 |a_{-\frac{d}{2}}(n+k)-a_{-\frac{d}{2}}(n)|^2
 \leq
 2|a_{-\frac{d}{2}}(n+k)|^2
 +2|a_{-\frac{d}{2}}(n)|^2.
\]
As $|n+k|\leq3K$ on this region, its contribution is at most
\[
 4\sum_{1\leq|n|\leq3K}|n|^{-d}
 \lesssim_d\log(2+K).
\]
If $|n|>2K$, then $|n+tk|\geq\frac{|n|}{2}$ for
$0\leq t\leq1$. The mean value theorem therefore gives
\[
 |a_{-\frac{d}{2}}(n+k)-a_{-\frac{d}{2}}(n)|
 \lesssim_d K|n|^{-\frac{d}{2}-1}.
\]
Hence
\[
 \sum_{|n|>2K}
 |a_{-\frac{d}{2}}(n+k)-a_{-\frac{d}{2}}(n)|^2
 \lesssim_d
 K^2\sum_{|n|>2K}|n|^{-d-2}
 \lesssim_d1.
\]
For $0<|k|<8$, the same tail estimate gives finiteness, while
the term $n=0$ gives strict positivity. Enlarging the constants
to include these finitely many $k$, we obtain
\begin{equation}\label{eq:rho-endpoint}
 \rho_{-\frac{d}{2}}(k)\approx_d\log(2+|k|),
 \qquad k\neq0.
\end{equation}

For the first assertion, suppose
$(\widehat x(k))_{k\neq0}\in\ell_2$.
Since the monomials $U^k$ are orthonormal, the sums
\[
 y_R=\sum_{0<|k|\leq R}\widehat x(k)U^k
\]
converge in $L_2(\qt)$ to an element $y$ with
\[
 \widehat y(0)=0,
 \qquad
 \widehat y(k)=\widehat x(k),\quad k\neq0.
\]
Then $x-\tau(x)-y\in L_1(\qt)$ has all Fourier coefficients
equal to zero, so Fourier uniqueness gives $x-\tau(x)=y$.
Conversely, if $x-\tau(x)\in L_2(\qt)$, Plancherel's identity
\eqref{Plancherel-qt} gives
\[
 \sum_{k\ne0}|\widehat x(k)|^2
 =\|x-\tau(x)\|_{L_2(\qt)}^2.
\]
The first estimate now follows from \eqref{eq:rho-strict}.
The second follows from \eqref{eq:rho-endpoint}, since
$\rho_{-\frac{d}{2}}(0)=0$ and
\[
 \log(2+|k|)\approx\log(1+|k|),
 \qquad k\neq0.
\]
\end{proof}

\begin{remark}\label{rem:Euclidean-comparison}
For $-d<\epsilon\leq-\frac{d}{2}$,
\cite[Lemma~3.3]{FSZ2024} shows that the commutator
$[(-\Delta)^{\frac{\epsilon}{2}},M_f]$ has no bounded extension
for any nonzero compactly supported $f\in L_1(\mathbb R^d)$.
This obstruction comes from the decay of the Riesz kernel
at infinity. On $\qt$, however, $I_\theta^\epsilon$ is compact
with norm one for every $\epsilon<0$. Hence, for
$x\in L_\infty(\qt)$,
\[
 [I_\theta^\epsilon,M_x]\in\mathcal K(L_2(\qt)),
 \qquad
 \|[I_\theta^\epsilon,M_x]\|_\infty\leq2\|x\|_\infty.
\]

The Sobolev upper estimate in Theorem~\ref{mta} fails even
when $\theta=0$ if
\[
d\geq3,\qquad 1-d<\epsilon\leq-\frac d2.
\]
See Example~\ref{ex:lower-order}.
\end{remark}

\begin{example}\label{ex:lower-order}
Let $d\geq3$ and $1-d<\epsilon\leq-\frac{d}{2}$.
There is a sequence
$(x_N)_{N\geq1}\subset C^\infty(\mathbb T^d)$ with
$\tau(x_N)=0$ such that
\[
 \sup_{N\geq1}
 \|x_N\|_{\dot W_{\frac{d}{1-\epsilon}}^1}<\infty,
 \qquad
 \|[I_0^\epsilon,M_{x_N}]\|_{\mathcal L_2}
 \longrightarrow\infty.
\]
There also exists
$x\in\dot W_{\frac{d}{1-\epsilon}}^1(\mathbb T^d)$
whose commutator has no Hilbert--Schmidt extension.
Thus the conclusions of Theorems~\ref{mta} and~\ref{mtb}
cannot extend to every element of
$\dot W_{\frac{d}{1-\epsilon}}^1(\mathbb T^d)$
in this range.
\end{example}

\begin{proof}
Identify $\mathbb T^d$ with
$\mathbb R^d/(2\pi\mathbb Z)^d$, equipped with normalized
Haar measure. Fix a nonzero
$\varphi\in C_c^\infty((-1,1)^d)$.
For $N\geq1$, set on $(-\pi,\pi]^d$
\[
 f_N(t)=N^{-\epsilon}\varphi(Nt),
 \qquad
 x_N=f_N-\tau(f_N).
\]
The function $f_N$ is zero near the boundary, so it defines
a smooth function on $\mathbb T^d$.
A change of variables gives
\[
 \big\|x_N \big\|_{\dot W_{\frac{d}{1-\epsilon}}^1}\approx1,
 \qquad
 \|x_N\|_2\approx N^{-\epsilon-\frac d2}.
\]
If $\epsilon<-\frac{d}{2}$, the second quantity tends to infinity.
Proposition~\ref{prop:HS} then gives
\[
 \|[I_0^\epsilon,M_{x_N}]\|_{\mathcal L_2}
 \longrightarrow\infty.
\]

At $\epsilon=-\frac{d}{2}$, for $k\neq0$,
\[
 \widehat{x_N}(k)
 =(2\pi)^{-d}N^{-\frac d2}
   \widehat\varphi\Big(\frac{k}{N}\Big),
 \qquad
 \widehat\varphi(\xi)
 =\int_{\mathbb R^d}\varphi(u)e^{-\ri u\cdot\xi}\,du.
\]
The function $\widehat\varphi$ is continuous and not
identically zero. Choose a closed cube
$K=\prod_{j=1}^d[a_j,b_j]$ of side length $\ell>0$,
disjoint from the origin, such that
$|\widehat\varphi|\geq c>0$ on $K$.
Put
\[
 r_0=\min_{\xi\in K}|\xi|>0.
\]
For all sufficiently large $N$,
\[
 \#\Big\{k\in\mathbb Z^d: \frac kN \in K\Big\}
 =\prod_{j=1}^d
   \bigl(\lfloor Nb_j\rfloor-\lceil Na_j\rceil+1\bigr)
 \geq\Big(\frac{\ell N}{2}\Big)^d.
\]
For these lattice points,
\[
 |k|\geq r_0N,
 \qquad
 |\widehat{x_N}(k)|^2
 \geq(2\pi)^{-2d}c^2N^{-d}.
\]
Therefore,
\begin{align*}
 \sum_{k\ne0}
 \log(1+|k|)|\widehat{x_N}(k)|^2
 &\geq(2\pi)^{-2d}c^2N^{-d}\log(1+r_0N)
       \#\Big\{k\in\mathbb Z^d: \frac kN \in K\Big\}\\
 &\gtrsim\log N.
\end{align*}
Using Proposition~\ref{prop:HS} again, we have
\[
 \|[I_0^\epsilon,M_{x_N}]\|_{\mathcal L_2}
 \longrightarrow\infty.
\]
Since $1<\frac{d}{1-\epsilon}<2$,
\[
 \|T\|_{\mathcal L_2}
 \leq C_{d,\epsilon}
       \|T\|_{\frac{d}{1-\epsilon},\infty}.
\]
Thus no uniform upper estimate can hold in terms of
$\|x_N\|_{\dot W_{\frac{d}{1-\epsilon}}^1}$.

We now show that some
$x\in\dot W_{\frac{d}{1-\epsilon}}^1(\mathbb T^d)$
has no Hilbert--Schmidt commutator extension.
Suppose otherwise. Then the linear map
\[
 W_{\frac{d}{1-\epsilon}}^1(\mathbb T^d)
 \longrightarrow\mathcal L_2,
 \qquad
 x\longmapsto[I_0^\epsilon,M_x],
\]
would be everywhere defined.
To check that its graph is closed, let
$x_j\to x$ in
$W_{\frac{d}{1-\epsilon}}^1(\mathbb T^d)$ and
$[I_0^\epsilon,M_{x_j}]\to T$ in $\mathcal L_2$.
The first convergence implies convergence in $L_1$, so
\[
 \widehat{x_j}(k)\longrightarrow\widehat{x}(k),
 \qquad k\in\mathbb Z^d.
\]
The second implies convergence of all matrix entries.
By \eqref{eq:matrix}, $T$ is a Hilbert--Schmidt extension
of $[I_0^\epsilon,M_x]$.
By the closed graph theorem, this map is bounded.
By the Poincar\'e inequality \eqref{eq:distribution-Poincare}, the mean-zero sequence $x_N$
is bounded in
$W_{\frac{d}{1-\epsilon}}^1(\mathbb T^d)$.
This contradicts the divergence of its commutator norms.

Finally, if a compact operator $T$ satisfies
\[
 \lim_{t\to\infty}
 t^{\frac{1-\epsilon}{d}}\mu(t,T)<\infty,
\]
then
$\mu(t,T)=O(t^{-\frac{1-\epsilon}{d}})$
as $t\to\infty$.
Since $\frac{d}{1-\epsilon}<2$, this implies
$T\in\mathcal L_2$.
Thus Theorem~\ref{mtb} cannot hold for the element $x$
obtained above.
\end{proof}

	\section{Proof of Theorem \ref{mta} }
	
	This section is devoted to the proof of Theorem \ref{mta}. The starting point is the following Cwikel type estimates quoted from \cite[Theorem~3.1]{MSX2019}.

	If $x\in L_p(\qt)$ and $g\in \ell_{p,\infty}(\mathbb{Z}^d)$ with $2< p <\infty$, then $M_x \, g(\nabla)$ is in $\mathcal{L}_{p,\infty} $ and
	\begin{equation}\label{ineq-cwikel}
		\|M_x\, g(\nabla)\|_{ \mathcal{L}_{p,\infty}} \leq c_{p,d} \|x\|_p \|g\|_{p,\infty}.
	\end{equation}
	
	The following lemma is easily deduced from \eqref{ineq-cwikel}.
	
	\begin{lem}\label{cwikel} Let $\alpha,\beta\in(0,\frac{d}{4})$ and let $p=\frac{d}{2(\alpha+\beta)}.$ For every $x\in L_p(\qt),$ we have
		$$\Big\|I_\theta^{-2\alpha}M_xI_\theta^{-2\beta}\Big\|_{\frac{d}{2(\alpha+\beta)},\infty}\leq c_{d,\alpha,\beta}\|x\|_{\frac{d}{2(\alpha+\beta)}}.$$
	\end{lem}
	\begin{proof} By polar decomposition, there is a partial isometry $u\in L_{\infty}(\qt)$ such that
		$$x=|x^{\ast}|^{\frac{\alpha}{\alpha+\beta}}u|x|^{\frac{\beta}{\alpha+\beta}}.$$
		We now write
		$$I_\theta^{-2\alpha}M_xI_\theta^{-2\beta}=I_\theta^{-2\alpha}M_{|x^{\ast}|^{\frac{\alpha}{\alpha+\beta}}}\cdot M_u\cdot M_{|x|^{\frac{\beta}{\alpha+\beta}}}I_\theta^{-2\beta}.$$
		By H\"older inequality \eqref{weak-type-holder}, we have
		$$\Big\|I_\theta^{-2\alpha}M_xI_\theta^{-2\beta}\Big\|_{\frac{d}{2(\alpha+\beta)},\infty}\leq c_{\frac{d}{2\alpha},\frac{d}{2\beta}}^{(1)}\Big\|I_\theta^{-2\alpha}M_{|x^{\ast}|^{\frac{\alpha}{\alpha+\beta}}}\Big\|_{\frac{d}{2\alpha},\infty}\Big\|M_{|x|^{\frac{\beta}{\alpha+\beta}}}I_\theta^{-2\beta}\Big\|_{\frac{d}{2\beta},\infty}.$$
		By Cwikel estimate \eqref{ineq-cwikel}, we have
		$$\Big\|I_\theta^{-2\alpha}M_{|x^{\ast}|^{\frac{\alpha}{\alpha+\beta}}}\Big\|_{\frac{d}{2\alpha},\infty}\leq c_{d,\alpha}^{(2)}\||x^{\ast}|^{\frac{\alpha}{\alpha+\beta}}\|_{\frac{d}{2\alpha}}=c_{d,\alpha}^{(2)}\|x\|_{\frac{d}{2(\alpha+\beta)}}^{\frac{\alpha}{\alpha+\beta}},$$
		$$\Big\|M_{|x|^{\frac{\beta}{\alpha+\beta}}}I_\theta^{-2\beta}\Big\|_{\frac{d}{2\beta},\infty}\leq c_{d,\beta}^{(2)}\||x|^{\frac{\beta}{\alpha+\beta}}\|_{\frac{d}{2\beta}}=c_{d,\beta}^{(2)}\|x\|_{\frac{d}{2(\alpha+\beta)}}^{\frac{\beta}{\alpha+\beta}}.$$
		Thus,
		$$\Big\|I_\theta^{-2\alpha}M_xI_\theta^{-2\beta}\Big\|_{\frac{d}{2(\alpha+\beta)},\infty}\leq c_{\frac{d}{2\alpha},\frac{d}{2\beta}}^{(1)}c_{d,\alpha}^{(2)}c_{d,\beta}^{(2)}\|x\|_{\frac{d}{2(\alpha+\beta)}}.$$
	\end{proof}
	
	Let us continue the proof of Theorem \ref{mta}. Note that $1\otimes \big[ I_\theta^{\epsilon}, M_x\big] =\big[ | \D|^{\epsilon},1\otimes M_x\big]$, so we are reduced to the estimate of $\big[ | \D|^{\epsilon},1\otimes M_x\big]$. The advantage of considering the latter commutator is that it can be easily expressed as the double operator integrals (see \cite{PSW2002} and \cite{PS2009}) acting on the following operators:
	
	\begin{defi} Let $\epsilon<1$ and let $x\in \dot{W}^{1}_{\frac{d}{1-\epsilon}}(\qt)$. Denote
		$$X_{\epsilon}=|\D|^{\frac{\epsilon-1}{2}}[\D,1\otimes M_x]|\D|^{\frac{\epsilon-1}{2}},$$
		$$Y_\epsilon=\sum_{l=1}^{\lceil\frac{1-\epsilon}{2}\rceil-1} \big(|\D|^{\epsilon+l-1}[\D,1\otimes M_x]|\D|^{-l}+|\D|^{-l}[\D,1\otimes M_x]|\D|^{\epsilon+l-1}\big).$$
	\end{defi}

	\begin{lem}\label{commutator-Sob-1}
		Let $x\in \dot{W}^{1}_{\frac{d}{1-\epsilon}}(\qt)$. If $\epsilon\in(-\frac{d}{2},1),$ then
		$$\big\|X_{\epsilon}\big\|_{\frac{d}{1-\epsilon},\infty},\big\|Y_{\epsilon}\big\|_{\frac{d}{1-\epsilon},\infty}   \leq B_{\epsilon,d}  \|x\|_{\dot W_{\frac{d}{1-\epsilon}}^1}.$$
	\end{lem}
	\begin{proof} Denote
		$$X_{\epsilon,j}=I_\theta^{\frac{\epsilon-1}{2}}M_{D_j x}I_\theta^{\frac{\epsilon-1}{2}},$$
		$$Y_{\epsilon,j,l}=I_\theta^{\epsilon+l-1}M_{D_jx}I_\theta^{-l}+I_\theta^{-l}M_{D_jx}I_\theta^{\epsilon+l-1},\quad 1\leq l<\lceil\frac{1-\epsilon}{2}\rceil.$$
		
		Since $\epsilon\in(1-d,1),$ it follows that
		$$0>\frac{\epsilon-1}{4}>-\frac{d}{4},\quad\frac{\epsilon-1}{4}+\frac{\epsilon-1}{4}=\frac{\epsilon-1}{2}.$$
		Applying Lemma \ref{cwikel} with $\alpha=\beta=\frac{1-\epsilon}{4},$ we conclude that
		$$\|X_{\epsilon,j}\|_{\frac{d}{1-\epsilon},\infty}\leq c_{d,\frac{1-\epsilon}{4},\frac{1-\epsilon}{4}}\|D_jx\|_{\frac{d}{1-\epsilon}}.$$
		Since $\epsilon\in(-\frac{d}{2},1),$ it follows that
		$$0>\frac{\epsilon+l-1}{2}>-\frac{d}{4},\quad 0>-\frac{l}{2}>-\frac{d}{4},\quad \frac{\epsilon+l-1}{2}+\frac{-l}{2}=\frac{\epsilon-1}{2},\quad 1\leq l<\lceil\frac{1-\epsilon}{2}\rceil.$$
		Applying Lemma \ref{cwikel} with $\alpha=-\frac{\epsilon+l-1}{2},$ $\beta=\frac{l}{2},$ we conclude that
		$$\|Y_{\epsilon,j,l}\|_{\frac{d}{1-\epsilon},\infty}\leq 4c_{d,-\frac{\epsilon+l-1}{2},\frac{l}{2}}\|D_jx\|_{\frac{d}{1-\epsilon}}.$$
		
		Note that
		$$X_{\epsilon}=\sum_{1\leq j\leq d}   \gamma_j \otimes  X_{\epsilon,j},\quad Y_{\epsilon}=\sum_{1\leq j\leq d}\sum_{l=1}^{\lceil\frac{1-\epsilon}{2}\rceil-1}  \gamma_j \otimes  Y_{\epsilon,j,l}.$$
		By the triangle inequality, we have
		$$\|X_{\epsilon}\|_{\frac{d}{1-\epsilon},\infty}\leq c_{\frac{d}{1-\epsilon}}\sum_{1\leq j\leq d}\|X_{\epsilon,j}\|_{\frac{d}{1-\epsilon},\infty}\leq dc_{\frac{d}{1-\epsilon}}c_{d,\frac{1-\epsilon}{4},\frac{1-\epsilon}{4}}\sum_{j=1}^d\|D_jx\|_{\frac{d}{1-\epsilon}}\leq c_{d,\epsilon}\|x\|_{\dot{W}^1_{\frac{d}{1-\epsilon}}},$$
\begin{align*}
 \|Y_{\epsilon}\|_{\frac{d}{1-\epsilon},\infty}
 &\leq c_{\frac{d}{1-\epsilon}}
       \sum_{1\leq j\leq d}
       \sum_{l=1}^{\lceil\frac{1-\epsilon}{2}\rceil-1}
       \|Y_{\epsilon,j,l}\|_{\frac{d}{1-\epsilon},\infty}\\
 &\leq 4c_{\frac{d}{1-\epsilon}}
       \sum_{1\leq j\leq d}
       \sum_{l=1}^{\lceil\frac{1-\epsilon}{2}\rceil-1}
       c_{d,-\frac{\epsilon+l-1}{2},\frac{l}{2}}
       \|D_jx\|_{\frac{d}{1-\epsilon}}\\
 &\leq c_{d,\epsilon}\|x\|_{\dot{W}^1_{\frac{d}{1-\epsilon}}}.
\end{align*}
	\end{proof}
	
	Next, set
	\begin{equation}\label{Psi-ep}
		\Psi_\epsilon(\lambda,\mu)=\frac{|\lambda|^{\epsilon}-|\mu|^{\epsilon}}{\lambda-\mu }\cdot |\lambda|^{ \frac{1-\epsilon}{2}}|\mu|^{\frac{1-\epsilon}{2}},\quad \lambda,\mu\in\mathbb{R}.
	\end{equation}
 For a real valued function $\phi: \mathbb{R}\rightarrow \mathbb{R}$, denote
 $$\phi^{[1]}(\lambda, \mu)  =  \frac{\phi(\lambda)-\phi(\mu)}{\lambda-\mu}.$$

	\begin{lem}\label{bounded-Psi} Let $\epsilon\in[-1,1)$ and $1<p<\infty.$ For every self-adjoint (potentially unbounded) $A$, the operator $T_{\Psi_\epsilon}^A:\mathcal{L}_{p,\infty}\to \mathcal{L}_{p,\infty}$ is bounded.
	\end{lem}
	\begin{proof} Set $\psi(\lambda)=|\lambda|,$ $\lambda\in\mathbb{R},$ and
		$$\Phi_\epsilon(\lambda, \mu )=\frac{\lambda^{\epsilon}-\mu^{ \epsilon}}{\lambda-\mu}\cdot  \lambda^{\frac{1-\epsilon}{2}} \mu^{\frac{1-\epsilon}{2}},\quad \lambda,\mu>0.$$
		It is immediate that
		$$\Psi_{\epsilon}(\lambda,\mu)=\Phi_{\epsilon}(|\lambda|,|\mu|)\cdot\psi^{[1]}(\lambda,\mu),\quad \lambda,\mu\in\mathbb{R}.$$
		We, therefore, have
		$$T_{\Psi_\epsilon}^A=T_{\Phi_\epsilon}^{|A|}\circ T_{\psi^{[1]}}^A.$$
		
		Let $\epsilon\in(-1,1).$ Set
		$$f_{\epsilon}(t)=\frac{t^{\epsilon}-1}{t-1}\cdot t^{\frac{1-\epsilon}{2}},\quad t>0.$$
		Clearly, $f_{\epsilon}\circ\exp$ is a Schwartz function. Taking into account that
		$$\Phi_{\epsilon}(\lambda,\mu)=f_{\epsilon}(\frac{\lambda}{\mu}),\quad\lambda,\mu>0,$$
		we infer from Lemma 9 in \cite{PS2009} that $T_{\Phi_\epsilon}^{|A|}:\mathcal{L}_{\infty}\to \mathcal{L}_{\infty}$ is bounded. By duality and interpolation, $T_{\Phi_\epsilon}^{|A|}:\mathcal{L}_{p,\infty}\to \mathcal{L}_{p,\infty}$ is bounded.
		
		If $\epsilon=-1,$ then $\Phi_{\epsilon}=-1$ and we immediately infer that $T_{\Phi_\epsilon}^{|A|}:\mathcal{L}_{p,\infty}\to \mathcal{L}_{p,\infty}$ is bounded.
		
		It is well known (see e.g. \cite{PS2011}) that $T_{\psi^{[1]}}^A:\mathcal{L}_{p,\infty}\to \mathcal{L}_{p,\infty}$ is bounded.
		
		Combining all these facts, we complete the proof.
	\end{proof}

When $\epsilon<-1$, the operator discussed in the above lemma is not bounded anymore. We need the following decomposition for the function $\Psi_\epsilon$ introduced in \eqref{Psi-ep}.
	
	\begin{lem}\label{Psi-ep-decomp}
	 Let $\epsilon<-1$ and let $\delta=\epsilon-2+2\lceil\frac{1-\epsilon}{2}\rceil.$ We have
		$$|\lambda|^{\epsilon}-|\mu|^{\epsilon}=\Psi_{\delta}(\lambda,\mu)\cdot F_{\epsilon}(\lambda,\mu)-\psi^{[1]}(\lambda,\mu)\cdot G_{\epsilon}(\lambda,\mu),$$
		where
		$$F_{\epsilon}(\lambda,\mu)=|\lambda|^{\frac{\epsilon-1}{2}}|\mu|^{\frac{\epsilon-1}{2}}(\lambda-\mu),\quad\lambda,\mu\in\mathbb{R},$$
		$$G_{\epsilon}(\lambda,\mu)=\Big(\sum_{l=1}^{\lceil\frac{1-\epsilon}{2}\rceil-1}\big(|\lambda|^{\epsilon+l-1}|\mu|^{-l}+|\lambda|^{-l}|\mu|^{\epsilon+l-1}\big)\Big)(\lambda-\mu),\quad \lambda,\mu\in\mathbb{R}.$$
	\end{lem}
	\begin{proof} We have
		$$1=t^{2n}\cdot t^{-n}-(t-1)\sum_{l=1}^nt^{l-1},\quad 1=t^{-n}+(t-1)\sum_{l=1}^nt^{-l},\quad t>0.$$	
		Now, set $n=\lceil\frac{1-\epsilon}{2}\rceil-1$ and note that $\delta=\epsilon+2n.$ Multiplying the first equality by $t^{\epsilon}$ and subtracting them, we obtain	
		$$t^{\epsilon}-1=(t^{\delta}-1)  \cdot t^{ \frac{1-\delta}{2}}\cdot t^{\frac{\epsilon-1}{2}}-(t-1)\cdot \Big(\sum_{l=1}^{\lceil\frac{1-\epsilon}{2}\rceil-1}\big(t^{\epsilon+l-1}+t^{-l}\big)\Big),\quad t>0.$$
		Setting $t=\frac{|\lambda|}{|\mu|},$ we obtain	
		$$|\lambda|^{\epsilon}-|\mu|^{\epsilon}=(|\lambda|^{\delta}-|\mu|^{\delta})  \cdot |\lambda|^{ \frac{1-\delta}{2}}|\mu|^{\frac{1-\delta}{2}}\cdot |\lambda|^{\frac{\epsilon-1}{2}}|\mu|^{\frac{\epsilon-1}{2}}-$$
		$$-(|\lambda|-|\mu|)\cdot \Big(\sum_{l=1}^{\lceil\frac{1-\epsilon}{2}\rceil-1}\big(|\lambda|^{\epsilon+l-1}|\mu|^{-l}+|\lambda|^{-l}|\mu|^{\epsilon+l-1}\big)\Big).$$
		Dividing and multiplying the right hand side by $\lambda-\mu,$ we complete the proof.
	\end{proof}
	
	Under the convention that $G_{\epsilon}(\lambda,\mu)=0$ when $\epsilon\geq -1$, we may unify the decomposition for $|\lambda|^{\epsilon}-|\mu|^{\epsilon}$ for all $\epsilon<1$.

	We are now able to conclude the proof of Theorem \ref{mta}.
	
	\begin{proof}[Proof of Theorem \ref{mta}]
Let $ \delta=\epsilon-2+2\lceil\frac{1-\epsilon}{2}\rceil$, then $\delta\in[-1,1)$. So $ T^{\D}_{\Psi_{\delta}}:\mathcal{L}_{\frac{d}{1-\epsilon},\infty}\to \mathcal{L}_{\frac{d}{1-\epsilon},\infty}$ is bounded by Lemma \ref{bounded-Psi}. Thus,
$$\|T^{\D}_{\Psi_{\delta}}(X_{\epsilon})-T^{\D}_{\psi^{[1]}}(Y_{\epsilon})\|_{\frac{d}{1-\epsilon},\infty}\leq c_{d,\epsilon}\|x\|_{\dot{W}^{1}_{\frac{d}{1-\epsilon}}(\qt)}.$$
Next, we claim that
			$$[|\D|^{\epsilon},1\otimes M_x]=T^{\D}_{\Psi_{\delta}}(X_{\epsilon})-T^{\D}_{\psi^{[1]}}(Y_{\epsilon})$$
			on their common domain $\mathbb{C}^{\mathbf{d}}  \otimes \mathcal{S}(\qt)$. With this claim, we immediately conclude the proof of Theorem \ref{mta}.

			We are reduced to proving the claim, i.e. to proving the fact that
					$$[|\D|^{\epsilon},1\otimes M_x] \mathbf{e}\otimes \xi=\big(T^{\D}_{\Psi_{\delta}}(X_{\epsilon})-T^{\D}_{\psi^{[1]}}(Y_{\epsilon})\big) \mathbf{e}\otimes\xi$$
		in $\mathbb{C}^{\mathbf{d}}  \otimes \mathcal{S}'(\qt)$	for any  $\mathbf{e}\otimes\xi \in\mathbb{C}^{\mathbf{d}}  \otimes \mathcal{S}(\qt)$. Recall that
		$$X_{\epsilon}  =T_{F_{\epsilon}}^{\D}([\D,1\otimes M_x]), \quad \mbox{and} \quad  Y_{\epsilon}=T^{\D}_{G_{\epsilon}}([\D,1\otimes M_x]).$$
		So for any $\mathbf{f}\otimes\eta \in\mathbb{C}^{\mathbf{d}}  \otimes \mathcal{S}(\qt)$, we have
		$$\Big\langle    \big(T^{\D}_{\Psi_{\delta}}(X_{\epsilon})-T^{\D}_{\psi^{[1]}}(Y_{\epsilon})\big) \mathbf{e}\otimes \xi  ,   \mathbf{f}\otimes\eta   \Big\rangle    =\Big\langle  T^\D_{\Psi_{\delta}F_{\epsilon}-\psi^{[1]}G_{\epsilon}}([\D,1\otimes M_x])   \mathbf{e}\otimes \xi  ,   \mathbf{f}\otimes\eta  \Big\rangle$$
			But since $\xi, \eta \in \mathcal{S}(\qt)$ and $\Psi_{\delta}F_{\epsilon}-\psi^{[1]}G_{\epsilon}= |\lambda|^{\epsilon}-|\mu|^{\epsilon}$ by Lemma \ref{Psi-ep-decomp}, we deduce from the definition of double operator integral that
			$$\Big\langle  T^\D_{\Psi_{\delta}F_{\epsilon}-\psi^{[1]}G_{\epsilon}}([\D,1\otimes M_x])   \mathbf{e}\otimes \xi  ,   \mathbf{f}\otimes\eta  \Big\rangle =  \langle \mathbf{e} , \mathbf{f}\rangle  \Big(   \langle x\xi,   I_\theta^{\epsilon}\eta  \rangle  -       \langle x I_\theta^{\epsilon}\xi,  \eta  \rangle    \Big).$$
			This finishes the proof of the claim, whence the proof of Theorem \ref{mta}.
	\end{proof}

	\section{Spectral asymptotics of pseudodifferential operators}\label{sec:PDO-Weyl}

In Section~\ref{subsec:Cstar-Weyl}, we prove spectral asymptotics
for operators with continuous principal symbols.
In Section~\ref{subsec:classical-Weyl}, we apply these results
to selfadjoint classical pseudodifferential operators
of negative order.

\subsection{Spectral asymptotics in the \texorpdfstring{$C^*$}{C*}-algebraic framework}
\label{subsec:Cstar-Weyl}

We use the $C^*$-algebra $\Pi= \Pi(C(\mathbb{T}_\theta^d),C(\mathbb{S}^{d-1}))$ and the principal symbol map
$\sym$ introduced in Section~\ref{subsec:Cstar-PDO}.
The results below extend the order $-1$ asymptotic formula
of \cite[Theorem~1.1]{SXZ2023} to arbitrary negative orders.

	\begin{thm}\label{key-psd}
		Let $d\geq 2$ and $p>0.$ If $T\in \Pi(C(\mathbb{T}_\theta^d), C(\mathbb{S}^{d-1}))$ is positive, then
		\begin{equation}\label{key-psd eq}
			\lim_{t\to\infty}t^{\frac{p}{d}}\mu(t,T^pJ_\theta^{-p})=\big(d^{-\frac{1}{d}}\|{\rm sym}(T)\|_{L_d(L_{\infty}(\mathbb{T}_\theta^d)\bar{\otimes} L_\infty(\mathbb{S}^{d-1}))}\big)^p.
		\end{equation}
	\end{thm}
	
	\begin{cor}\label{key-psd cor} Let $d\geq 2$ and $p>0.$ If $T\in \Pi(C(\mathbb{T}_\theta^d), C(\mathbb{S}^{d-1})),$ then
		$$\lim_{t\to\infty}t^{\frac{1}{p}}\mu(t,TJ_\theta^{-\frac{d}{p}})=d^{-\frac{1}{p}}\|{\rm sym}(T)\|_{L_p(L_{\infty}(\mathbb{T}_\theta^d)\bar{\otimes} L_\infty(\mathbb{S}^{d-1}))}.$$
	\end{cor}

Let us recall the following lemma due to Birman and Solomyak \cite{BS1987}.
\begin{lem}\label{weak weyl}
Let $p>0$, let $A\in\mathcal L_{p,\infty}$, and let
$B\in(\mathcal L_{p,\infty})_0$.
If $\lim\limits_{t\to\infty}t^{\frac{1}{p}}\mu(t,A)$ exists, then
\[
\lim_{t\to\infty}t^{\frac{1}{p}}\mu(t,A+B)
=
\lim_{t\to\infty}t^{\frac{1}{p}}\mu(t,A).
\]
\end{lem}

Corollary~\ref{commutator-PSD-dp} also holds with
$I_\theta^{-\alpha}$ replaced by $J_\theta^{-\alpha}$:
\begin{equation}\label{eq:Bessel-small-commutator}
 [T,J_\theta^{-\alpha}]
 \in(\mathcal L_{\frac{d}{\alpha},\infty})_0,\qquad T\in\Pi,\quad\alpha>0.
\end{equation}
The Fourier multiplier
\[
 J_\theta^{-\alpha}-I_\theta^{-\alpha}
\]
has symbol $O((1+|n|)^{-\alpha-2})$ on $\mathbb Z^d$.
Consequently,
\[
 J_\theta^{-\alpha}-I_\theta^{-\alpha}
 \in\mathcal L_{\frac{d}{\alpha+2},\infty}
 \subset(\mathcal L_{\frac{d}{\alpha},\infty})_0.
\]
Together with Corollary~\ref{commutator-PSD-dp}, this proves
\eqref{eq:Bessel-small-commutator}.

	To prove Theorem \ref{key-psd}, we first state the following result established in \cite[Proposition~6.2]{FSZ2024}.

	\begin{thm}\label{FSZ separable part thm} Let $A\in\mathcal{L}_{\infty}$ and $B\in\mathcal{L}_{d,\infty}$ be positive operators. If $1<p<d$ is such that $[B^p,A^{\frac{1}{2}}]\in(\mathcal{L}_{\frac{d}{p},\infty})_0,$ then
		$$B^pA^p-(A^{\frac{1}{2}}BA^{\frac{1}{2}})^p\in(\mathcal{L}_{\frac{d}{p},\infty})_0.$$
	\end{thm}

By \cite[Theorem~1.1]{SXZ2023} and Lemma~\ref{weak weyl},
we have the following formula, since
$J_\theta^{-1}-I_\theta^{-1}\in(\mathcal L_{d,\infty})_0$.

\begin{thm}\label{SXZ1 main thm} Let $d\geq 2.$ If $T\in \Pi(C(\mathbb{T}_\theta^d), C(\mathbb{S}^{d-1}))$, then
		$$\lim_{t\to\infty}t^{\frac{1}{d}}\mu(t,TJ_\theta^{-1})= d^{-\frac{1}{d}}\|{\rm sym}(T)\|_{L_d(L_{\infty}(\mathbb{T}_\theta^d)\bar{\otimes} L_\infty(\mathbb{S}^{d-1}))}.$$
	\end{thm}
	
	\begin{lem}\label{xiao step 1 lemma} Equality~\eqref{key-psd eq} holds for every $1<p<d$ and for every positive $T\in \Pi(C(\mathbb{T}_\theta^d), C(\mathbb{S}^{d-1})).$
	\end{lem}
	\begin{proof}
Put $B=J_\theta^{-1}$ and
$C=T^{\frac{1}{2}}BT^{\frac{1}{2}}$.
By \eqref{eq:Bessel-small-commutator},
\[
[B,T^{\frac{1}{2}}]\in(\mathcal L_{d,\infty})_0.
\]
Hence
\[
C-TB
=T^{\frac{1}{2}}[B,T^{\frac{1}{2}}]
\in(\mathcal L_{d,\infty})_0.
\]
Theorem~\ref{SXZ1 main thm} and Lemma~\ref{weak weyl} give
\[
\lim_{t\to\infty}t^{\frac{1}{d}}\mu(t,C)
=d^{-\frac{1}{d}}\|\sym(T)\|_d.
\]
Since $C\geq0$, it follows that
\[
\lim_{t\to\infty}t^{\frac{p}{d}}\mu(t,C^p)
=\bigl(d^{-\frac{1}{d}}\|\sym(T)\|_d\bigr)^p.
\]

Moreover, \eqref{eq:Bessel-small-commutator} gives
\[
[B^p,T^{\frac{1}{2}}]
\in(\mathcal L_{\frac{d}{p},\infty})_0.
\]
Applying Theorem~\ref{FSZ separable part thm} with $A=T$,
we obtain
\[
B^pT^p-C^p
\in(\mathcal L_{\frac{d}{p},\infty})_0.
\]
Applying Lemma~\ref{weak weyl} once more, we obtain
\[
\lim_{t\to\infty}t^{\frac{p}{d}}\mu(t,B^pT^p)
=\bigl(d^{-\frac{1}{d}}\|\sym(T)\|_d\bigr)^p.
\]
Finally, $(B^pT^p)^*=T^pB^p$, so
$\mu(t,B^pT^p)=\mu(t,T^pB^p)$.
This proves \eqref{key-psd eq}.
\end{proof}
	
	\begin{lem}\label{xiao step 2 lemma} Let $p>0$ and let $T\in\Pi(C(\mathbb{T}_\theta^d), C(\mathbb{S}^{d-1}))$ be positive. Equality~\eqref{key-psd eq} holds for $p$ if and only if it holds for $2p.$
	\end{lem}
	\begin{proof} Suppose \eqref{key-psd eq} holds with $2p.$ Thus,
		$$\lim_{t\to\infty}t^{\frac{2p}{d}}\mu(t,T^{2p}J_\theta^{-2p})=\big(d^{-\frac{1}{d}}\|{\rm sym}(T)\|_d \big)^{2p}.$$
		Since $[T^{2p},J_\theta^{-p}]\in(\mathcal{L}_{\frac{d}{p},\infty})_0$ by \eqref{eq:Bessel-small-commutator}, it follows that
		$$T^{2p}J_\theta^{-2p}-J_\theta^{-p}T^{2p}J_\theta^{-p}\in(\mathcal{L}_{\frac{d}{2p},\infty})_0.$$
		Hence,
		$$\lim_{t\to\infty}t^{\frac{2p}{d}}\mu(t,J_\theta^{-p}T^{2p}J_\theta^{-p})=\big(d^{-\frac{1}{d}}\|{\rm sym}(T)\|_d \big)^{2p}.$$
		Taking into account that
		$$\mu(t,J_\theta^{-p}T^{2p}J_\theta^{-p})=\mu^2(t,T^pJ_\theta^{-p}),$$
		we infer that \eqref{key-psd eq} holds for $p.$
		
		Conversely, if \eqref{key-psd eq} holds for $p$, then
the identity above gives the corresponding limit for
$J_\theta^{-p}T^{2p}J_\theta^{-p}$.
The same remainder estimate and Lemma~\ref{weak weyl}
then give \eqref{key-psd eq} for $2p$.
	\end{proof}
	
	\begin{proof}[Proof of Theorem~\ref{key-psd}]
Let $T\in\Pi(C(\qt),C(\mathbb S^{d-1}))$ be positive.
Equality \eqref{key-psd eq} holds for $p=1$ by
Theorem~\ref{SXZ1 main thm}, and for $1<p<2$ by
Lemma~\ref{xiao step 1 lemma}, since $d\geq2$.

For $p>0$, choose $k\in\mathbb Z$ with $1\leq2^kp<2$.
Lemma~\ref{xiao step 2 lemma} reduces the claim for $p$
to the claim for $2^kp$, which has already been proved.
\end{proof}
	
	\begin{proof}[Proof of Corollary \ref{key-psd cor}] Suppose first $T$ is positive. Using Theorem \ref{key-psd} with $T^{\frac{1}{p}},$ we obtain
		$$\lim_{t\to\infty}t^{\frac{p}{d}}\mu(t,TJ_\theta^{-p})=d^{-\frac{p}{d}}\|{\rm sym}(T)\|_{L_{\frac{d}{p}}(L_{\infty}(\mathbb{T}_\theta^d)\bar{\otimes} L_\infty(\mathbb{S}^{d-1}))}.$$	
		Replacing $p$ with $\frac{d}{p},$ we complete the proof for positive $T.$
		
		The assertion for general $T$ follows from the equality
		$$\mu(t,TJ_\theta^{-\frac{d}{p}})=\mu(t,|T|J_\theta^{-\frac{d}{p}}),\quad t>0.$$
	\end{proof}

For a compact selfadjoint operator $K$, let $\lambda_j^\pm(K)$,
$j\in\mathbb N_0$, denote its positive eigenvalues and the absolute
values of its negative eigenvalues, respectively, in nonincreasing
order, with multiplicities and zeros appended to finite lists.
Write
\[
 N_\pm(s;K)=\rTr\mathbf1_{(s,\infty)}(\pm K),\qquad
 N(s;|K|)=\rTr\mathbf1_{(s,\infty)}(|K|),\quad s>0,
\]
where $\rTr$ is the ordinary Hilbert space trace.

\begin{thm}\label{thm:signed-sandwich-Weyl}
Let $m>0$, $p=\frac{d}{m}$, and $A=A^*\in\Pi$. Set
$a=\sym(A)$ and $K=J_\theta^{-\frac{m}{2}}AJ_\theta^{-\frac{m}{2}}$. Then
\begin{equation}\label{eq:signed-sandwich-Weyl}
 \lim_{j\to\infty}j^{\frac{1}{p}}\lambda_j^\pm(K)
 =\left(\frac{1}{d}\int_{\mathbb S^{d-1}}
                  \tau(a(s)_\pm^p)\,ds\right)^{\frac{1}{p}},
\end{equation}
where $ds$ is the usual unnormalized surface measure.
\end{thm}

\begin{proof}
Put $B=J_\theta^{-\frac{m}{2}}\in\mathcal L_{2p,\infty}$, so
$K=K^*\in\mathcal L_{p,\infty}$. By
\eqref{eq:small-ideal-product} and \eqref{eq:Bessel-small-commutator},
\begin{equation}\label{eq:sandwich-total-remainder}
 K-AJ_\theta^{-m}=[B,A]B\in(\mathcal L_{p,\infty})_0.
\end{equation}
Thus Corollary~\ref{key-psd cor} gives
\[
 \lim_{t\to\infty}t^{\frac{1}{p}}\mu(t,K)
 =\left(\frac{1}{d}\int_{\mathbb S^{d-1}}\tau(|a(s)|^p)\,ds\right)^{\frac{1}{p}}.
\]
Continuous functional calculus gives $A_\pm\in\Pi$ and
$\sym(A_\pm)=a_\pm$. Define
\[
 Q_\pm=BA_\pm B,\qquad
 c_\pm=\frac{1}{d}\int_{\mathbb S^{d-1}}\tau(a(s)_\pm^p)\,ds.
\]
Applying \eqref{eq:sandwich-total-remainder} to $A_\pm$ gives
$t^{\frac{1}{p}}\mu(t,Q_\pm)\to c_\pm^{\frac{1}{p}}$.
For a nonincreasing sequence $(\lambda_j)_{j\in\mathbb N_0}$,
put $N(s)=\#\{j:\lambda_j>s\}$.
Then $N(s)>j$ if and only if $\lambda_j>s$.
Together with $|a|^p=a_+^p+a_-^p$, this gives
\begin{equation}\label{eq:three-counting-limits}
 s^pN(s;|K|)\longrightarrow c_++c_-,
 \qquad s^pN(s;Q_\pm)\longrightarrow c_\pm
 \quad(s\downarrow0).
\end{equation}
This equivalence includes zero constants: for any nonnegative
decreasing sequence, $\mu_j=o(j^{-\frac{1}{p}})$ if and only if
$\#\{j:\mu_j>s\}=o(s^{-p})$, by comparing
$\mu_j>s$ with $j<\#\{k:\mu_k>s\}$.

As $K=Q_+-Q_-$, the min--max principle \cite{BS1987} gives
\[
 N_\pm(s;K)\leq N(s;Q_\pm),\qquad
 \limsup_{s\downarrow0}s^pN_\pm(s;K)\leq c_\pm.
\]
On the other hand,
$N(s;|K|)=N_+(s;K)+N_-(s;K)$, so
\[
 \liminf_{s\downarrow0}s^pN_+(s;K)
 \geq c_++c_--\limsup_{s\downarrow0}s^pN_-(s;K)
 \geq c_+.
\]
Interchanging the signs proves
$s^pN_\pm(s;K)\to c_\pm$. Inverting once more gives
\eqref{eq:signed-sandwich-Weyl}.
\end{proof}

\subsection{Spectral asymptotics for traditional classical pseudodifferential operators}
\label{subsec:classical-Weyl}\label{app:classical-Weyl}

The preceding results allow continuous principal symbols.
For traditional classical pseudodifferential operators, the homogeneous
principal symbols are smooth and belong to a proper dense
subalgebra of the continuous symbol algebra. We establish this
connection and deduce \cite[Conjecture~7.1]{MP2023} from
Theorem~\ref{key-psd} and Corollary~\ref{key-psd cor}.

The Fr\'echet topology of $\mathcal S(\qt)$ is defined by
$p_N(x)=\max_{|\alpha|_1\leq N}\|D^\alpha x\|_\infty$.
We use the symbol classes in
\cite[Definitions~3.1 and~3.12]{HLP2019II},
with the above normalization of $D_j$.

\begin{defi}\label{def:usual-symbols}
For $r\in\mathbb R$, the class
$S^r(\mathbb R^d;\mathcal S(\qt))$ consists of smooth maps
$\sigma:\mathbb R^d\to\mathcal S(\qt)$ satisfying
\[
 \|D^\alpha\partial_\xi^\beta\sigma(\xi)\|_\infty
 \leq C_{\alpha,\beta}(1+|\xi|^2)^{\frac{r-|\beta|_1}{2}}
 \quad(\alpha,\beta\in\mathbb N_0^d).
\]
By \cite[Proposition~5.8]{HLP2019II}, its quantization is
\begin{equation}\label{eq:usual-quantization}
 P_\sigma x=\sum_{n\in\mathbb Z^d}\widehat x(n)\sigma(n)U^n,
 \qquad x\in\mathcal S(\qt),
\end{equation}
where the series converges in $\mathcal S(\qt)$ and the symbol
acts by left multiplication. These operators form $\Psi^r(\qt)$.

A symbol is classical if it has smooth homogeneous components
$\sigma_{r-j}:\mathbb R^d\setminus\{0\}\to\mathcal S(\qt)$,
$\sigma_{r-j}(t\xi)=t^{r-j}\sigma_{r-j}(\xi)$ for $t>0$, such that
\begin{equation}\label{eq:classical-expansion}
 \sigma-\sum_{j=0}^{N-1}\chi\sigma_{r-j}
 \in S^{r-N}(\mathbb R^d;\mathcal S(\qt)),\qquad N\geq1.
\end{equation}
Here $\chi$ is smooth, vanishes near $0$, and equals one for
$|\xi|\geq1$. We write $\Psi_{\mathrm{cl}}^r(\qt)$ for this
operator class and $\sigma_r(P)$ for the principal symbol.
\end{defi}

The class $\Psi_{\mathrm{cl}}^r(\qt)$ is denoted by
$\Psi^r(\qt)$ in \cite{MP2023}.
Operators of order zero are bounded by
\cite[Proposition~10.1]{HLP2019II}.
By \cite[Propositions~7.6 and~8.3]{HLP2019II},
composition and adjoints satisfy
\begin{equation}\label{eq:classical-principal-calculus}
 \sigma_{r+s}(PQ)=\sigma_r(P)\sigma_s(Q),\qquad
 \sigma_r(P^*)=\sigma_r(P)^*.
\end{equation}
By \cite[Proposition~5.14]{HLP2019II},
$J_\theta^r\in\Psi_{\mathrm{cl}}^r(\qt)$ has
principal symbol $|\xi|^r$.
If $r>0$ and $R\in\Psi^{-r}(\qt)$, then
$R=(RJ_\theta^r)J_\theta^{-r}\in\mathcal L_{\frac{d}{r},\infty}$.
The first factor is bounded, and
\[
\#\{n\in\mathbb Z^d:(1+|n|^2)^{-\frac r2}>s\}
\leq C_{d,r}s^{-\frac dr},\qquad 0<s<1.
\]

Let $q_M$ be the usual $C^M$-seminorms on $\mathbb S^{d-1}$.
In particular,
$q_0(g)=\|g\|_\infty$.
The projective tensor product
$\mathcal S(\qt)\widehat\otimes_\pi C^\infty(\mathbb S^{d-1})$
is the completion for the seminorms
\[
 \gamma_{N,M}(u)
 =\inf_{u=\sum_jx_j\otimes g_j}
       \sum_jp_N(x_j)q_M(g_j),
 \qquad N,M\in\mathbb N_0,
\]
where the infimum is over all finite tensor representations.
On $C^\infty(\mathbb S^{d-1};\mathcal S(\qt))$, convergence
means uniform convergence of all coordinate derivatives in
every seminorm $p_N$, on the same compact patches.
Under the canonical identification
\[
 C(\qt)\otimes_{\min}C(\mathbb S^{d-1})
 \cong C(\mathbb S^{d-1};C(\qt)),
\]
a finite tensor $u=\sum_jx_j\otimes g_j$ corresponds to
the function $s\mapsto\sum_jg_j(s)x_j$.
The norm on the right is the uniform norm, so
\[
 \|u\|_{\min}
 =\sup_{s\in\mathbb S^{d-1}}
       \Big\|\sum_jg_j(s)x_j\Big\|_\infty.
\]

\begin{prop}\label{prop:classical-in-Pi}
There is a continuous, dense, proper inclusion
\begin{equation}\label{eq:strict-symbol-inclusion}
 \mathcal S(\qt)\widehat\otimes_\pi C^\infty(\mathbb S^{d-1})
 \ \subsetneq\
 C(\qt)\otimes_{\min}C(\mathbb S^{d-1}).
\end{equation}
Every $P\in\Psi_{\mathrm{cl}}^0(\qt)$ belongs to
$\Pi=\Pi(C(\qt),C(\mathbb S^{d-1}))$, and
$\sym(P)=\sigma_0(P)|_{\mathbb S^{d-1}}$.
\end{prop}

\begin{proof}
The completed projective tensor product is canonically
$C^\infty(\mathbb S^{d-1};\mathcal S(\qt))$ as a Fr\'echet
space \cite[Chapters~44, 50--51]{Treves1967}.
The minimal tensor product is
$C(\mathbb S^{d-1};C(\qt))$ with the uniform norm.
The natural inclusion is therefore injective.
It is continuous since $\|u\|_{\min}\leq\gamma_{0,0}(u)$.
It is dense because Fourier polynomials are dense in $C(\qt)$
and smooth functions are dense in $C(\mathbb S^{d-1})$.
To see that the inclusion is proper, consider $a(s)=|s_1|$.
This function belongs to $C(\mathbb S^{d-1};C(\qt))$.
If it were a smooth $\mathcal S(\qt)$-valued function,
then $\tau(a(s))=|s_1|$ would be smooth on the sphere.
This is impossible, since $|s_1|$ is not differentiable
at points where $s_1=0$.

For the coordinate functions $s_j$, we have
$\pi_2(s_j)=D_jI_\theta^{-1}$; see \cite[p.~21]{SXZ2023}.
Hence $P_0=1-\sum_{j=1}^d(D_jI_\theta^{-1})^2\in\Pi$.
The Fourier matrix units $M_{U^m}P_0M_{U^n}^*$,
$m,n\in\mathbb Z^d$, belong to $\Pi$.
Their linear span is norm dense in $\mathcal K(L_2(\qt))$,
so $\mathcal K(L_2(\qt))\subset\Pi$.

Now set $a=\sigma_0(P)|_{\mathbb S^{d-1}}$. Approximate $a$ by
finite tensors $\sum_jx_j\otimes g_j$ in the projective topology.
The estimate
\[
\Big\|\sum_jM_{x_j}\pi_2(g_j)\Big\|_{\Bc(L_2(\qt))}
\leq\gamma_{0,0}\Big(\sum_jx_j\otimes g_j\Big)
\]
shows that the finite sums converge in operator norm to $T_a\in\Pi$.
Moreover, $\sym(T_a)=a$ and
$T_aU^n=a(\frac{n}{|n|})U^n$ for $n\ne0$.
By \eqref{eq:classical-expansion},
$P-P_{\chi\sigma_0(P)}\in\Psi^{-1}(\qt)$ is compact.
The operator $P_{\chi\sigma_0(P)}-T_a$ has rank at most one.
Thus $P-T_a$ is compact, so $P\in\Pi$ and $\sym(P)=a$.
\end{proof}

\begin{cor}\label{cor:McDonald-Ponge-conjecture}
Let $P=P^*\in\Psi_{\mathrm{cl}}^{-m}(\qt)$, $m>0$, have
principal symbol $\rho_{-m}$, and put $p=\frac{d}{m}$. Then
\begin{equation}\label{eq:McDonald-Ponge-conjecture}
 \lim_{j\to\infty}j^{\frac{1}{p}}\lambda_j^\pm(P)
 =\left(\frac{1}{d}\int_{\mathbb S^{d-1}}
                 \tau(\rho_{-m}(s)_\pm^p)\,ds\right)^{\frac{1}{p}}.
\end{equation}
This proves \cite[Conjecture~7.1]{MP2023} without an ellipticity
assumption.
\end{cor}

\begin{proof}
Set $A=J_\theta^{\frac m2}PJ_\theta^{\frac m2}$ on
$\mathcal S(\qt)$.
By \eqref{eq:classical-principal-calculus}, $A$ has order zero
and principal symbol $|\xi|^m\rho_{-m}(\xi)$.
Its bounded extension is selfadjoint, since $A$ is symmetric
on the dense subspace $\mathcal S(\qt)$.
Proposition~\ref{prop:classical-in-Pi} gives $A\in\Pi$ and
$\sym(A)(s)=\rho_{-m}(s)$. The identity
$P=J_\theta^{-\frac{m}{2}}AJ_\theta^{-\frac{m}{2}}$ holds first on
$\mathcal S(\qt)$ and then on $L_2(\qt)$ by boundedness.
Theorem~\ref{thm:signed-sandwich-Weyl} proves the result.
\end{proof}

The larger symbol algebra in \eqref{eq:strict-symbol-inclusion}
is used when passing from $A$ to $A_\pm$. Smoothness is not
preserved by this operation: $a(s)=s_1$ is smooth, whereas
$a_+(s)=\max\{s_1,0\}$ is not. Continuous functional calculus gives $A_\pm\in\Pi$.

	\section{An approximation expression for commutator}

In this section, we approximate the commutator
$[I_\theta^\epsilon,M_x]$ by its leading term
for $x\in\mathcal S(\mathbb T_\theta^d)$ and $\epsilon<1$.
The difference between the commutator and this leading term
belongs to $(\mathcal L_{\frac{d}{1-\epsilon},\infty})_0$.
Hence the two operators have the same singular value asymptotics.
We will combine this result
with the spectral asymptotics in Section~\ref{sec:PDO-Weyl}
to prove Theorem~\ref{mtb}.
	
	\begin{thm}\label{approximation theorem}
		Let $d\geq 2$ and $\epsilon<1$. For every $x\in \mathcal{S}(\mathbb{T}_\theta ^d)$, denote
		$$T=\sum_{1\leq j \leq d }M_{D_j x }D_j I_\theta^{-1} .$$
		We have
		$$[I_\theta^{\epsilon},M_x] -\epsilon T  J_\theta^{\epsilon-1} \in(\mathcal{L}_{\frac{d}{1-\epsilon},\infty})_0.$$
	\end{thm}
	
In the classical (commutative) setting, Theorem
\ref{approximation theorem} follows from standard
pseudodifferential operator calculus. For smooth $x$, $M_x$
is a pseudodifferential operator of order $0$, while
$J_0^{\epsilon}$ is of order $\epsilon$.
The asymptotic expansion of the commutator
$[J_0^{\epsilon},M_x]$ shows that its
principal symbol coincides with that of
$\epsilon \sum_{j=1}^d
    M_{D_j x} D_j I_0^{-1}
    J_0^{\epsilon-1}.$

Here we give a functional-calculus proof of Theorem
\ref{approximation theorem} on quantum tori, following
the approach in \cite{MSX2020,FSZ2024}.
	
	\begin{lem}\label{second summand lemma}
		Let $\epsilon\in(0,1).$ For $x \in \mathcal{S}(\qt)$ we have
		$$\int_0^{\infty}\frac{1}{1+\lambda-\Delta}  \Big[   \Delta, [\Delta, M_x]    \Big] \frac{1}{(1+\lambda-\Delta)^2}\cdot \lambda^{\frac{\epsilon}{2}}d\lambda\in\mathcal{L}_{d,\infty}.$$
	\end{lem}
	\begin{proof} We write
$$[\Delta,M_x]=-\sum_{k=1}^d[D_k^2,M_x]=-\sum_{k=1}^d \big(D_k[D_k,M_x]+[D_k,M_x]D_k\big)=$$
		$$=-\sum_{k=1}^d \big(D_kM_{D_kx}+M_{D_kx}D_k\big)=-2\sum_{k=1}^dM_{D_kx}D_k-\sum_{k=1}^d[D_k,M_{D_kx}]=M_{\Delta x}-2\sum_{k=1}^dM_{D_kx}D_k.$$
		Thus,
		$$[\Delta,[\Delta,M_x]]=[\Delta,M_{\Delta x}]-2\sum_{k=1}^d[\Delta,M_{D_kx}]D_k=4\sum_{k,l=1}^dM_{D_kD_lx}D_kD_l-4\sum_{k=1}^dM_{D_k\Delta x}D_k+M_{\Delta^2x}.$$

We clearly have
\begin{align*}
 \Big\|\frac{1}{1+\lambda-\Delta}M_{D_kD_lx}
          \frac{D_kD_l}{(1+\lambda-\Delta)^2}\Big\|_{d,\infty}
 &\leq\Big\|\frac{1}{1+\lambda-\Delta}\Big\|_{d,\infty}
       \|M_{D_kD_lx}\|_{\infty}
       \Big\|\frac{D_kD_l}{(1+\lambda-\Delta)^2}\Big\|_{\infty}\\
 &\leq(1+\lambda)^{-\frac{1}{2}}\|J_\theta^{-1}\|_{d,\infty}
       \cdot\|x\|_{W^2_{\infty}(\qt)}\cdot(1+\lambda)^{-1},
\end{align*}
\begin{align*}
 \Big\|\frac{1}{1+\lambda-\Delta}M_{D_k\Delta x}
          \frac{D_k}{(1+\lambda-\Delta)^2}\Big\|_{d,\infty}
 &\leq\Big\|\frac{1}{1+\lambda-\Delta}\Big\|_{d,\infty}
       \|M_{D_k\Delta x}\|_{\infty}
       \Big\|\frac{D_k}{(1+\lambda-\Delta)^2}\Big\|_{\infty}\\
 &\leq(1+\lambda)^{-\frac{1}{2}}\|J_\theta^{-1}\|_{d,\infty}
       \cdot\|x\|_{W^3_{\infty}(\qt)}
       \cdot(1+\lambda)^{-\frac{3}{2}},
\end{align*}
and
\begin{align*}
 \Big\|\frac{1}{1+\lambda-\Delta}M_{\Delta^2x}
          \frac{1}{(1+\lambda-\Delta)^2}\Big\|_{d,\infty}
 &\leq\Big\|\frac{1}{1+\lambda-\Delta}\Big\|_{d,\infty}
       \|M_{\Delta^2x}\|_{\infty}
       \Big\|\frac{1}{(1+\lambda-\Delta)^2}\Big\|_{\infty}\\
 &\leq(1+\lambda)^{-\frac{1}{2}}\|J_\theta^{-1}\|_{d,\infty}
       \cdot\|x\|_{W^4_{\infty}(\qt)}\cdot(1+\lambda)^{-2}.
\end{align*}
		Thus,
		$$\Big\|\frac{1}{1+\lambda-\Delta}  \Big[   \Delta, [\Delta, M_x]    \Big] \frac{1}{(1+\lambda-\Delta)^2}\Big\|_{d,\infty}\leq \frac{9}{(1+\lambda)^{\frac{3}{2}}}\|J_\theta^{-1}\|_{d,\infty}\cdot\|x\|_{W^4_{\infty}(\qt)}$$
		and
		$$\int_0^{\infty}\Big\|\frac{1}{1+\lambda-\Delta}  \Big[   \Delta, [\Delta, M_x]    \Big] \frac{1}{(1+\lambda-\Delta)^2}\cdot \lambda^{\frac{\epsilon}{2}}\Big\|_{d,\infty}d\lambda<\infty.$$
		This yields the assertion.
	\end{proof}
	
	\begin{proof}[Proof of Theorem \ref{approximation theorem} for $0<\epsilon<1$] By the functional calculus, we have (see, e.g., \cite[Theorem 1, p. 232]{BS1987}) 
		\begin{equation}\label{power-functional-cal eq0}
			J_\theta^{\epsilon}=\frac{\sin(\frac{\pi\epsilon}{2})}{\pi}\int_0^{\infty}\frac{1-\Delta}{1+\lambda-\Delta}\lambda^{\frac{\epsilon}{2}-1}d\lambda.\end{equation}
		Thus,
		$$[J_\theta^{\epsilon},M_x]=  \frac{\sin(\frac{\pi\epsilon}{2})}{\pi}\int_0^{\infty} \big[  \frac{1-\Delta}{1+\lambda-\Delta},M_x\big]   \lambda^{\frac{\epsilon}{2}-1}d\lambda.$$
		It is immediate that
		$$\big[\frac{1-\Delta}{1+\lambda-\Delta},M_x\big]=\big[1-\frac{\lambda}{1+\lambda-\Delta},M_x\big]=-\big[\frac{1}{1+\lambda-\Delta},M_x\big]\cdot\lambda.$$
		Hence,
		$$[J_\theta^{\epsilon},M_x]=-\frac{\sin(\frac{\pi\epsilon}{2})}{\pi}\int_0^{\infty}\big[\frac{1}{1+\lambda-\Delta},M_x   \big]\lambda^{\frac{\epsilon}{2}}d\lambda.$$
		For $x \in \mathcal{S}(\qt),$ we have
		$$\big[\frac{1}{1+\lambda-\Delta},M_x\big]=[\Delta, M_x] \frac{1}{(1+\lambda-\Delta)^{2}}+\frac{1}{1+\lambda-\Delta}   \Big[   \Delta, [\Delta, M_x]    \Big]\frac{1}{(1+\lambda-\Delta)^2}.$$
		Thus,
\begin{align*}
[J_\theta^{\epsilon},M_x]
={}&-[\Delta,M_x]\cdot \frac{\sin(\frac{\pi\epsilon}{2})}{\pi}\int_0^{\infty}\frac{1}{(1+\lambda-\Delta)^2}\lambda^{\frac{\epsilon}{2}}d\lambda\\
&-\frac{\sin(\frac{\pi\epsilon}{2})}{\pi}\int_0^{\infty}\frac{1}{1+\lambda-\Delta}  \Big[   \Delta, [\Delta, M_x]\Big] \frac{1}{(1+\lambda-\Delta)^2}\cdot \lambda^{\frac{\epsilon}{2}}d\lambda.
\end{align*}
		It follows from Lemma \ref{second summand lemma} that
		$$[J_\theta^{\epsilon},M_x]+[\Delta,M_x]\cdot \frac{\sin(\frac{\pi\epsilon}{2})}{\pi}\int_0^{\infty}\frac{1}{(1+\lambda-\Delta)^2}\lambda^{\frac{\epsilon}{2}}d\lambda\in\mathcal{L}_{d,\infty}.$$
		By the functional calculus, we have
		$$\frac{\sin(\frac{\pi\epsilon}{2})}{\pi}\int_0^{\infty}\frac{1}{(1+\lambda-\Delta)^2}\lambda^{\frac{\epsilon}{2}}d\lambda=\frac{\epsilon}{2}J_\theta^{\epsilon-2}.$$
		Thus,
		$$[J_\theta^{\epsilon},M_x]+\frac{\epsilon}{2}[\Delta,M_x]J_\theta^{\epsilon-2}\in\mathcal{L}_{d,\infty}.$$
		Since
		$$J_\theta^{\epsilon}-I_\theta^{\epsilon}\in\mathcal{L}_{d,\infty},\quad [\Delta,M_x]J_\theta^{\epsilon-2}+2TJ_\theta^{\epsilon-1}\in\mathcal{L}_{d,\infty},$$
		the assertion follows.
	\end{proof}

	\begin{proof}[Proof of Theorem \ref{approximation theorem} for $\epsilon<0$]
		
		Fix $m\in\mathbb{N}$ and $\delta\in(0,1)$ such that $\epsilon=-m\delta.$ Using the equality
		$$[A,B^{-m}]=\sum_{l=0}^{m-1}B^{-l}[A,B^{-1}]B^{l-m+1}=-\sum_{l=0}^{m-1}B^{-l-1}[A,B]B^{l-m},$$
		we write
		$$[J_\theta^{\epsilon},M_x]=[J_\theta^{-m\delta},M_x]=-\sum_{l=0}^{m-1}J_\theta^{-(l+1)\delta}[J_\theta^{\delta},M_x]J_\theta^{(l-m)\delta}.$$
		Using Theorem \ref{approximation theorem} for $\delta,$ we write
		$$[J_\theta^{\delta},M_x]-\delta TJ_\theta^{\delta-1}\in(\mathcal L_{\frac{d}{1-\delta},\infty})_0.$$
		By H\"older inequality \eqref{weak-type-holder},
\begin{align*}
 &J_\theta^{-(l+1)\delta}[J_\theta^{\delta},M_x]
       J_\theta^{(l-m)\delta}-\delta J_\theta^{-(l+1)\delta}
       \cdot TJ_\theta^{\delta-1}\cdot J_\theta^{(l-m)\delta}\\
= &J_\theta^{-(l+1)\delta}
       \cdot\Big([J_\theta^{\delta},M_x]-\delta TJ_\theta^{\delta-1}\Big)
       \cdot J_\theta^{(l-m)\delta}\\
\in  &\mathcal{L}_{\frac{d}{(l+1)\delta},\infty}
       \cdot(\mathcal{L}_{\frac{d}{1-\delta},\infty})_0
       \cdot\mathcal{L}_{\frac{d}{(m-l)\delta},\infty}\\
\subset  &(\mathcal{L}_{\frac{d}{1+m\delta},\infty})_0
       =(\mathcal{L}_{\frac{d}{1-\epsilon},\infty})_0.
\end{align*}
		Thus,
		$$[J_\theta^{\epsilon},M_x]+\delta\sum_{l=0}^{m-1}J_\theta^{-(l+1)\delta}TJ_\theta^{(l-m+1)\delta-1}\in(\mathcal{L}_{\frac{d}{1-\epsilon},\infty})_0.$$
		
		By \eqref{eq:Bessel-small-commutator},
		$$[J_\theta^{-(l+1)\delta},T]\in(\mathcal{L}_{\frac{d}{(l+1)\delta},\infty})_0.$$
		Thus,
		$$J_\theta^{-(l+1)\delta}TJ_\theta^{(l-m+1)\delta-1}-TJ_\theta^{-m\delta-1}\in(\mathcal{L}_{\frac{d}{1-\epsilon},\infty})_0.$$
		Consequently,
		$$[J_\theta^{\epsilon},M_x]+m\delta TJ_\theta^{-m\delta-1}\in(\mathcal{L}_{\frac{d}{1-\epsilon},\infty})_0.$$
		Since
		$$J_\theta^{\epsilon}-I_\theta^{\epsilon}\in(\mathcal{L}_{\frac{d}{1-\epsilon},\infty})_0,$$
		the assertion follows.
	\end{proof}

	\section{Proofs of the main theorems}
	
In this section we prove Theorems~\ref{mtb} and~\ref{mtd}.

		The next lemma shows that the right hand side norm in \eqref{mtb-eq} is indeed equivalent to the homogeneous Sobolev norm of $x\in \dot{W}^{1}_{\frac{d}{1-\epsilon}}(\qt)$.
	
			\begin{lem}\label{mtc}
		Let $1\leq p < \infty $. There exists a strictly positive constant $c_{d,p}$ such that, for every $x\in \dot{W}^{1}_{p}(\qt)$,
		$$c_{d,p}\|x\|_{\dot{W}^{1}_{p}(\qt)} \leq  \|  \sum_{1\leq j \leq d }D_jx\otimes s_j \|_{L_{p}(L_{\infty}(\mathbb{T}_\theta^d)\bar{\otimes} L_\infty(\mathbb{S}^{d-1}))} \leq d^{1-\frac{1}{p}}\Big(\int_{\mathbb{S}^{d-1}}|s_1|^p d s\Big)^{\frac{1}{p}} \|x\|_{\dot{W}^{1}_{p}(\qt)}.$$
	\end{lem}

	\begin{proof} 
	The second inequality in Lemma \ref{mtc} follows immediately from the triangle inequality. In order to prove the first inequality, for a fixed $x\in  \dot{W}^1_p(\qt)$, we assume without loss of generality that
		$$\|D_1 x \|_p  \geq \|D_j x\|_p ,\quad  2\leq j \leq d.$$
		Let
		$$\Omega = \{s\in \mathbb{S}^{d-1}  : |s_1| \geq \sqrt{\frac{d^2}{d^2+1}}\}.$$
		We have
\begin{align*}
 \|\sum_{j=1}^dD_jx\otimes s_j\|^p_{L_p(L_{\infty}(\mathbb{T}_\theta^d)\bar{\otimes} L_\infty(\mathbb{S}^{d-1}))}
 &=\int_{\mathbb{S}^{d-1}}
   \|\sum_{j=1}^ds_jD_jx\|^p_{L_p(L_{\infty}(\mathbb{T}_\theta^d))}\,ds\\
 &\geq\int_{\Omega}
   \|\sum_{j=1}^ds_jD_jx\|^p_{L_p(L_{\infty}(\mathbb{T}_\theta^d))}\,ds.
\end{align*}
		If $s\in\Omega,$ then it follows from the triangle inequality that
\begin{align*}
 \|\sum_{j=1}^ds_jD_jx\|_{L_p(L_{\infty}(\mathbb{T}_\theta^d))}
 &\geq\|s_1D_1x\|_{L_p(L_{\infty}(\mathbb{T}_\theta^d))}
       -\sum_{j=2}^d\|s_jD_jx\|_{L_p(L_{\infty}(\mathbb{T}_\theta^d))}\\
 &\geq\frac{d}{\sqrt{d^2+1}}
       \|D_1x\|_{L_p(L_{\infty}(\mathbb{T}_\theta^d))}
       -\frac{1}{\sqrt{d^2+1}}
       \sum_{j=2}^d\|D_jx\|_{L_p(L_{\infty}(\mathbb{T}_\theta^d))}\\
 &\geq\frac{d}{\sqrt{d^2+1}}
       \|D_1x\|_{L_p(L_{\infty}(\mathbb{T}_\theta^d))}
       -\frac{1}{\sqrt{d^2+1}}
       \sum_{j=2}^d\|D_1x\|_{L_p(L_{\infty}(\mathbb{T}_\theta^d))}\\
 &=\frac{1}{\sqrt{d^2+1}}
       \|D_1x\|_{L_p(L_{\infty}(\mathbb{T}_\theta^d))}.
\end{align*}
		Thus,
		$$\|\sum_{j=1}^dD_jx\otimes s_j\|_{L_p(L_{\infty}(\mathbb{T}_\theta^d)\bar{\otimes} L_\infty(\mathbb{S}^{d-1}))}\geq \frac{1}{\sqrt{d^2+1}}\|D_1x\|_{L_p(L_{\infty}(\mathbb{T}_\theta^d))}\cdot m(\Omega)^{\frac{1}{p}}\geq\frac{d^{-\frac{1}{p}}m(\Omega)^{\frac{1}{p}}}{\sqrt{d^2+1}}\|x\|_{\dot{W}^1_p}.$$
	\end{proof}

	\subsection{Proof of Theorem \ref{mtb}}
	
	Before proceeding with the proof of Theorem \ref{mtb}, let us recall the following lemma due to Birman and Solomyak \cite{BS1987}.
	
	\begin{lem}\label{bs limit lemma}
		Let $(A_n)_{n\geq0}\subset\mathcal{L}_{p,\infty}$ be such that
		\begin{enumerate}[{\rm (i)}]
			\item $A_n\to A$ in $\mathcal{L}_{p,\infty};$
			\item for every $n\geq0,$ there exists the limit
			$$\lim_{t\to\infty}t^{\frac{1}{p}}\mu(t,A_n)=c_n.$$
		\end{enumerate}
		Then the following limits exist and are equal,
		$$\lim_{t\to\infty}t^{\frac{1}{p}}\mu(t,A)=\lim_{n\to\infty}c_n.$$
	\end{lem}
	
	\begin{thm}\label{thm:asymp-smooth}
 Let $d\geq 2$ and let $\epsilon<1.$ For every $x\in\mathcal{S}(\qt),$ we have
		$$\lim_{t\to\infty}t^{\frac{1-\epsilon}{d}}\mu\Big(t,[I_\theta^{\epsilon},M_x]\Big)=|\epsilon|d^{\frac{\epsilon-1}{d}}\|\sum_{j=1}^dD_jx\otimes s_j\|_{L_{\frac{d}{1-\epsilon}}(L_{\infty}(\mathbb{T}_\theta^d)\bar{\otimes} L_\infty(\mathbb{S}^{d-1}))}.$$
	\end{thm}
	\begin{proof} Set
		$$T=\sum_{1\leq j \leq d }M_{D_j x }D_j I_\theta^{-1} .$$
		By Corollary \ref{key-psd cor},
		$$ \lim_{t\to\infty}t^{\frac{1-\epsilon}{d}}\mu\Big(t,T  J_\theta^{\epsilon-1} \Big)=d^{ \frac{\epsilon-1}{d}}    \|{\rm sym}(T)\|_{L_{\frac{d}{1-\epsilon}}(L_{\infty}(\mathbb{T}_\theta^d)\bar{\otimes} L_\infty(\mathbb{S}^{d-1}))}.$$
		Clearly,
		$${\rm sym}(T)=\sum_{j=1}^dD_jx\otimes s_j.$$
		Thus,
		$$\lim_{t\to\infty}t^{\frac{1-\epsilon}{d}}\mu\Big(t,T  J_\theta^{\epsilon-1} \Big)=d^{\frac{\epsilon-1}{d}}\|\sum_{j=1}^dD_jx\otimes s_j\|_{L_{\frac{d}{1-\epsilon}}(L_{\infty}(\mathbb{T}_\theta^d)\bar{\otimes} L_\infty(\mathbb{S}^{d-1}))}.$$
		The assertion follows from Lemma \ref{weak weyl} and Theorem \ref{approximation theorem}.
	\end{proof}

	We are going to conclude Theorem \ref{mtb} for arbitrary $x\in\dot{W}^{1}_{\frac{d}{1-\epsilon}}(\qt)$ using an approximation argument.
	
	\begin{proof}[Proof of Theorem \ref{mtb}] Thanks to \cite[Proposition~4.7]{XXY2018}, we may find a sequence $\{x_n\}_{n\geq 1}\subset\mathcal S(\qt)$ such that $x_n\to x$ in $\dot W^1_{\frac{d}{1-\epsilon}}(\qt)$. For these $\{x_n\}_{n\geq 1}  $, we already have
		$$\lim_{t\to\infty}t^{\frac{1-\epsilon}{d}}\mu\Big(t,[I_\theta^{\epsilon},M_{x_n}]\Big) =     |\epsilon| \, d^{ \frac{\epsilon-1}{d}}    \|  \sum_{1\leq j \leq d }D_j x_n\otimes s_j \|_{L_{\frac{d}{1-\epsilon}}(L_{\infty}(\mathbb{T}_\theta^d)\bar{\otimes} L_\infty(\mathbb{S}^{d-1}))}.$$
		Using Lemma \ref{mtc}, we obtain
		$$ \lim_{n \to \infty}   \|  \sum_{1\leq j \leq d }D_j x_n\otimes s_j \|_{L_{\frac{d}{1-\epsilon}}(L_{\infty}(\mathbb{T}_\theta^d)\bar{\otimes} L_\infty(\mathbb{S}^{d-1}))} =   \|  \sum_{1\leq j \leq d }D_j x\otimes s_j \|_{L_{\frac{d}{1-\epsilon}}(L_{\infty}(\mathbb{T}_\theta^d)\bar{\otimes} L_\infty(\mathbb{S}^{d-1}))}.$$
		On the other hand, since $x_n \to  x $ in $\dot{W}^{1}_{\frac{d}{1-\epsilon}}(\qt)$, it follows from Theorem \ref{mta} that
		$$[I_\theta^{\epsilon},M_{x_n}] \to [I_\theta^{\epsilon},M_{x }],\quad n\to\infty,$$
		in $\mathcal{L}_{\frac{d}{1-\epsilon},\infty}.$ Thus, by Lemma \ref{bs limit lemma}, we have
		$$ \lim_{n \to \infty} \lim_{t\to\infty}t^{\frac{1-\epsilon}{d}}\mu\Big(t,[I_\theta^{\epsilon},M_{x_n}]\Big)    = \lim_{t\to\infty}t^{\frac{1-\epsilon}{d}}\mu\Big(t,[I_\theta^{\epsilon},M_{x }]\Big) $$
		Combining the above equalities, we complete the proof.
	\end{proof}

	\subsection{Proof of Theorem \ref{mtd}}
	\begin{proof}[Proof of Theorem \ref{mtd}]

For $t\in\mathbb T^d$, let
$T_tU^n=e^{\ri t\cdot n}U^n$, $n\in\mathbb Z^d$.
The automorphism $T_t$ preserves $\tau$ and extends to a
unitary on $L_2(\qt)$.	

		If $\psi\in L_1(\mathbb{T}^d)$ and $x\in L_1(\qt),$	then we set
		$$\psi\ast x=\int_{\mathbb{T}^d}\psi(t)T_tx dt.$$
		Clearly,
		$$M_{T_tx}=T_tM_xT_{-t},\quad t\in\mathbb{T}^d.$$
		Thus,
		$$M_{\psi\ast x}=\int_{\mathbb{T}^d}\psi(t)T_tM_xT_{-t}dt$$
		and
		$$[I_\theta^{\epsilon},M_{\psi\ast x}]=\int_{\mathbb{T}^d}\psi(t)T_t[I_\theta^{\epsilon},M_x]T_{-t}dt.$$
		
				Let now $x\in L_1(\qt)$ be such that $[I_\theta^{\epsilon},M_x]\in\mathcal{L}_{\frac{d}{1-\epsilon},\infty}.$ Since $\frac{d}{1-\epsilon}>1,$ it follows that
	\begin{align*}
		\|[I_\theta^{\epsilon},M_{\psi\ast x}]\|_{\frac{d}{1-\epsilon},\infty}
		&\leq c_{\frac{d}{1-\epsilon}}\int_{\mathbb{T}^d}\|T_t[I_\theta^{\epsilon},M_x]T_{-t}\|_{\frac{d}{1-\epsilon},\infty}|\psi(t)|dt \\
		&=c_{\frac{d}{1-\epsilon}}\|[I_\theta^{\epsilon},M_x]\|_{\frac{d}{1-\epsilon},\infty}\|\psi\|_1.
	\end{align*}
		If $\psi\in C^{\infty}(\mathbb{T}^d),$ then $\psi\ast x\in\mathcal{S}(\qt).$ It follows that
		$$\limsup_{t\to\infty}t^{\frac{1-\epsilon}{d}}\mu\Big(t,[I_\theta^{\epsilon},M_{\psi\ast x}]\Big)\leq \|[I_\theta^{\epsilon},M_{\psi\ast x}]\|_{\frac{d}{1-\epsilon},\infty}\leq c_{\frac{d}{1-\epsilon}}\|[I_\theta^{\epsilon},M_x]\|_{\frac{d}{1-\epsilon},\infty}\|\psi\|_1.$$
		Lemma \ref{mtc} and Theorem \ref{thm:asymp-smooth} yield
		$$\|\psi\ast x\|_{\dot{W}_{\frac{d}{1-\epsilon}}^1}\leq c_{d,\epsilon}^{(1)}\|[I_\theta^{\epsilon},M_x]\|_{\frac{d}{1-\epsilon},\infty}\|\psi\|_1.$$
		
		Let $\psi_N$ be the $N$-th Fej\'er kernel (it is a trigonometric polynomial, hence a smooth function). It follows that
		$$\|\psi_N\ast x\|_{\dot{W}_{\frac{d}{1-\epsilon}}^1}\leq c_{d,\epsilon}^{(1)}\|[I_\theta^{\epsilon},M_x]\|_{\frac{d}{1-\epsilon},\infty}.$$
		In particular, for every $1\leq j\leq d,$
		$$\|D_j(\psi_N\ast x)\|_{\frac{d}{1-\epsilon}}\leq c_{d,\epsilon}^{(1)}\|[I_\theta^{\epsilon},M_x]\|_{\frac{d}{1-\epsilon},\infty}.$$

		Hence, for each $1\leq j\leq d,$ we obtain a bounded sequence $\{D_j(\psi_N\ast x)\}_{N\in \mathbb{N}}$ in $L_{ \frac{d}{1-\epsilon}}(\qt).$ Since $L_{ \frac{d}{1-\epsilon}}(\qt)$ is reflexive, we may assume that $D_j(\psi_N\ast x)$ converges weakly to some $y_j \in L_{\frac{d}{1-\epsilon}}(\qt).$ Thus, $D_j(\psi_N\ast x)\to y_j$ in the sense of distributions. On the other hand, $\psi_N\ast x\to x$ in $L_1(\qt)$ and, therefore, in the sense of distributions. Thus, $D_j(\psi_N\ast x)\to D_jx$ in the sense of distributions. By uniqueness of the limit, $y_j=D_jx$ for every $1\leq j\leq d.$ Since $y_j \in L_{\frac{d}{1-\epsilon}}(\qt)$ for every $1\leq j\leq d,$ it follows that $D_jx\in L_{\frac{d}{1-\epsilon}}(\qt)$ for every $1\leq j\leq d.$ Thus, $x\in\dot{W}^1_{\frac{d}{1-\epsilon}}(\qt).$
		
Passing to a common subsequence for $1\leq j\leq d$,
the weak convergence above and weak lower semicontinuity
give
\[
 \|x\|_{\dot W_{\frac{d}{1-\epsilon}}^1(\qt)}
 \leq
 \liminf_{N\to\infty}
 \|\psi_N*x\|_{\dot W_{\frac{d}{1-\epsilon}}^1(\qt)}
 \leq
 c_{d,\epsilon}^{(1)}
 \|[I_\theta^{\epsilon},M_x]\|_{\frac{d}{1-\epsilon},\infty}.
\]
This proves Theorem~\ref{mtd}.
	\end{proof}

Combining Theorem~\ref{mtd} with
Proposition~\ref{prop:HS}, we obtain the following
necessary conditions in the lower range.

\begin{cor}\label{cor:lower-range-necessary-conditions}
Let
\[
d\geq3,
\qquad
1-d<\epsilon\leq-\frac{d}{2},
\]
and let $x\in L_1(\qt)$. If
\[
 [I_\theta^\epsilon,M_x]\in\mathcal L_{\frac{d}{1-\epsilon},\infty},
\]
then $x\in\dot W_{\frac{d}{1-\epsilon}}^1(\qt)$ and
\[
 \begin{cases}
 x\in L_2(\qt),
     &\epsilon<-\frac{d}{2},\\[1mm]
 \displaystyle
 \sum_{k\ne0}
 \log(1+|k|)|\widehat x(k)|^2<\infty,
     &\epsilon=-\frac{d}{2}.
 \end{cases}
\]
\end{cor}

\begin{proof}
For every $T\in\mathcal L_{\frac{d}{1-\epsilon},\infty}$,
\[
 \sum_{k\geq0}\mu(k,T)^2
 \leq
 \|T\|_{\frac{d}{1-\epsilon},\infty}^2
 \sum_{k\geq0}(k+1)^{-\frac{2(1-\epsilon)}{d}}
 <\infty,
\]
since $\frac{d}{1-\epsilon}<2$. Thus $\mathcal L_{\frac{d}{1-\epsilon},\infty}\subset\mathcal L_2$.
The Sobolev conclusion follows from Theorem~\ref{mtd}.
The remaining conclusions follow from Proposition~\ref{prop:HS}.
\end{proof}

\section*{Acknowledgements}
The research of Xinyu Zhang is supported by the National Natural Science Foundation
of China (Nos. 12371138 and W2441002). 
The author thanks Prof. Xiao Xiong for his guidance and helpful discussions.
He also thanks Prof. Fedor Sukochev and Prof. Dmitriy Zanin for their comments and suggestions.


\begin{thebibliography}{99}

		
 		\bibitem{Baaj1988} S. Baaj.
 		\newblock {Calcul pseudodiff\'erentiel et produits crois\'es de $C^*$-alg\`ebres, I and II}.
		\newblock {\em C. R. Acad. Sc. Paris, s\'er. I,} 307 (1988), 581-586 and 663-666.


		
		
		
		
		
		
 		\bibitem{BS1970} M. S. Birman, M. Z. Solomyak.
 		\newblock The principal term of the spectral asymptotics for ``nonsmooth"
 		elliptic problems.
 		\newblock {\em Funktsional. Anal. i Prilozhen.}, 4 (1970), 1-13.
		
		
		
 		\bibitem{BS1972} M. S. Birman, M. Z. Solomyak.
 		\newblock Spectral asymptotics of nonsmooth elliptic operators. I, II.
 		\newblock {\em Trudy Moskov Mat. Obsc}, 27 (1972), 3-52; 28 (1973), 3-34.
		
		
		
		
		\bibitem{BS1987} M. S. Birman, M. Z. Solomyak.
		\newblock {\em Spectral theory of selfadjoint operators in Hilbert space}.
		\newblock Mathematics and its Applications (Soviet Series). D. Reidel Publishing Co., Dordrecht, 1987.
		
		
		
		
		
		
		
		
		
		
 		\bibitem{Connes1980} A. Connes.
 		\newblock {\em $C^*$-alg\`ebres et g\'eom\'etrie diff\'erentielle}.
 		\newblock {\em C. R. Acad. Sc. Paris, s\'er. A,} 290 (1980), 599-604.
		
		
		
		\bibitem{Connes1988}A. Connes.
		\newblock The action functional in noncommutative geometry.
		\newblock {\em Comm. Math. Phys.}, 117 (1988), 673-683.
		
		
		
		
		
		
		
		
		
		\bibitem{CST1994} A. Connes, D. Sullivan and N. Teleman. \newblock Quasiconformal mappings, operators on Hilbert space, and local formulae for characteristic classes.
		\newblock {\em Topology}, 33 (1994), 663-681.
		
		\bibitem{CRW1976} R.~R. Coifman, R.~Rochberg and G.~Weiss.
\newblock Factorization theorems for Hardy spaces in several variables.
\newblock {\em Ann. of Math. (2)}, 103 (1976), 611-635.

		
		%
		
		
		
		
		
		
		
		
		
	
\bibitem{DPS2023} P.~G.~Dodds, B.~de Pagter and F.~A.~Sukochev.
\newblock {\em Noncommutative Integration and Operator Theory}.
\newblock Progress in Mathematics, vol.~349.
Birkh\"auser/Springer, Cham, 2023.



\bibitem{FSZ2023} R. L. Frank, F. Sukochev and D. Zanin.
		\newblock Asymptotics of singular values for quantum derivatives.
		\newblock {\em Trans. Amer. Math. Soc.}, 376 (2023), 2047-2088.	

		
		
		
		\bibitem{FSZ2024} R.~L. Frank, F.~Sukochev and D.~Zanin.
\newblock Endpoint Schatten class properties of commutators.
\newblock {\em Adv. Math.}, 450 (2024), 109738.
		
		
		
		
		
		
		
		
	

\bibitem{HLP2019I} H.~Ha, G.~Lee and R.~Ponge.
\newblock Pseudodifferential calculus on noncommutative tori, I.
Oscillating integrals.
\newblock {\em Internat. J. Math.}, 30 (2019), 1950033, 74 pp.

\bibitem{HLP2019II} H.~Ha, G.~Lee and R.~Ponge.
\newblock Pseudodifferential calculus on noncommutative tori, II.
Main properties.
\newblock {\em Internat. J. Math.}, 30 (2019), 1950034, 73 pp.	
		
		
		
		
		
		
 		\bibitem{Hor1965} L. H\"{o}rmander.
 		\newblock Pseudo-differential operators.
 		\newblock {\em Comm. Pure Appl. Math.} 18 (1965), 501–517.
		
		
		\bibitem{Janson1978} S.~Janson.
\newblock Mean oscillation and commutators of singular integral operators.
\newblock {\em Ark. Mat.}, 16 (1978), 263-270.

\bibitem{JP1988} S.~Janson and J.~Peetre.
\newblock Paracommutators---boundedness and Schatten--von Neumann properties.
\newblock {\em Trans. Amer. Math. Soc.}, 305 (1988), 467-504.

		\bibitem{JW1982} S. Janson and T.H. Wolff.
		\newblock Schatten classes and commutators of singular integral operators.
		\newblock {\em Ark. Mat.}, 20 (1982), 301-310.
		
		
 		\bibitem{KN1965}J. J. Kohn and L. Nirenberg.
 		\newblock An algebra of pseudo-differential operators.
 		\newblock {\em Comm. Pure Appl. Math.} 18 (1965), 269-305.
		
		
		
		
		
		
		
		
		
		\bibitem{LMSZ2017} S. Lord, E. McDonald, F. Sukochev and D. Zanin.
		\newblock {Quantum differentiability of essentially bounded functions on Euclidean space}.
		\newblock {\em J. Funct. Anal.}, 273 (2017), 2353-2387.
		
		\bibitem{LSZ2012} S. Lord, F. Sukochev and D. Zanin.
		\newblock {\em Singular traces: theory and applications}.
		\newblock Walter de Gruyter, Vol. 46, 2012.
		
		
		
		\bibitem{LSZ2021} S. Lord, F. Sukochev and D. Zanin.
		\newblock {\em Singular traces. Vol 1. Theory}.
		\newblock Walter de Gruyter, Vol. 46/1, 2021.
		
		
		
		
		
		
		
		
		
		
		
		

           \bibitem{MP2023} E.~McDonald and R.~Ponge.
           \newblock Dixmier trace formulas and negative eigenvalues of
Schr\"odinger operators on curved noncommutative tori.
           \newblock {\em Adv. Math.}, 412 (2023), 108815.
		
		
		
		\bibitem{MSX2019} E. McDonald, F. Sukochev and X. Xiong.
		\newblock Quantum differentiability on quantum tori.
		\newblock {\em Comm. Math. Phys.}, 371 (2019), 1231-1260.
		
		
		
		
		\bibitem{MSX2020} E. McDonald, F. Sukochev and X. Xiong.
		\newblock Quantum differentiability on noncommutative Euclidean spaces.
		\newblock {\em Comm. Math. Phys.}, 379 (2020), 491-542.
		
		
		
		
		
		\bibitem{MSZ2019} E. McDonald, F. Sukochev and D. Zanin.
		\newblock A $C^*$-algebraic approach to the principal symbol II.
		\newblock {\em Math. Ann.}, 374 (2019), 273-322.
		
		
		
		\bibitem{MSZ2022} E. McDonald, F. Sukochev and D. Zanin.
		\newblock A noncommutative Tauberian theorem and Weyl asymptotics in noncommutative geometry.
		\newblock {\em Lett. Math. Phys.}, 112:77, (2022), 45 pp.
		
		
		
		
		
		\bibitem{PSW2002}B.~de Pagter, F.~Sukochev and H.~Witvliet.
		\newblock Double operator integrals.
		\newblock {\em J. Funct. Anal.}, 192 (2002), 52-111.
		
		
		
		
		%
		%
		
		
		
		
		
		
		
		
		
		
		
		
		
		
		
		
		
		
		\bibitem{Ponge2023} R.~Ponge.
\newblock Connes' integration and Weyl's laws.
\newblock {\em J. Noncommut. Geom.}, 17 (2023), no. 2, 719-767.

		\bibitem{PS2009} D.~Potapov and F.~Sukochev.
		\newblock Unbounded Fredholm modules and double operator integrals.
		\newblock {\em J. Reine Angew. Math.}, 626 (2009) 159-185.
		
		
		
		
		\bibitem{PS2011} D.~Potapov and F.~Sukochev.
		\newblock Operator-Lipschitz functions in Schatten-von Neumann classes.
		\newblock {\em Acta Math.}, 207 (2011), 375-389.
		
		
		
		
		
		
		
		
		
		
		
		
		
		
		
		
		
		
		
		
		

		
		
           \bibitem{R2022}G. Rozenblum.
 		\newblock Eigenvalues of singular measures and Connes noncommutative integration.
		\newblock {\em J. Spectr. Theory}, 12 (2022), 259-300.
		
			
		\bibitem{Simon1979} B.~Simon.
		\newblock {\em Trace Ideals and Their Applications}.
		\newblock London Mathematical Society Lecture Note Series, vol.35, Cambridge University Press, Cambridge-New York, 1979.	

		
		
		
		%
		
		
		
		
		
		
 		\bibitem{Stein1993}E.~M.~Stein.
 		\newblock Harmonic analysis: real-variable methods, orthogonality, and oscillatory integrals.
 		\newblock Princeton Mathematical Series, 43. Monographs in Harmonic Analysis, III. Princeton University Press, Princeton, NJ, 1993.
		
		
		
		\bibitem{SZ2018} F.~Sukochev and D.~Zanin.
		\newblock A $C^*$-algebraic approach to the principal symbol I.
		\newblock  {\em J. Operator Theory}, 80 (2018), 481-522.
		
		

\bibitem{SZHolder2021} F.~Sukochev and D.~Zanin.
\newblock Optimal constants in non-commutative H\"older inequality
for quasi-norms.
\newblock {\em Proc. Amer. Math. Soc.}, 149 (2021), no.~9, 3813--3817.
\par


\bibitem{SXZ2023} F.~Sukochev, X. Xiong and D.~Zanin.
		\newblock Asymptotics of singular values for quantised derivatives on noncommutative tori.
		\newblock  {\em J. Funct. Anal.}, 285 (2023), 110021.
		
		
		
 \bibitem{Tao2018} J.~Tao.
 \newblock The theory of pseudo-differential operators on the noncommutative $n$-torus.
 \newblock {\em J. Phys. Conf. Ser.}, 965 (2018) 1-12.

           \bibitem{Treves1967} F.~Tr\`eves.
           \newblock {\em Topological Vector Spaces, Distributions and Kernels}.
           \newblock Academic Press, New York-London, 1967.

		
		
		\bibitem{Uchiyama1978} A.~Uchiyama.
\newblock On the compactness of operators of Hankel type.
\newblock {\em Tohoku Math. J. (2)}, 30 (1978), 163-171.

		\bibitem{XXY2018} X. Xiong, Q. Xu and Z. Yin.
		\newblock Sobolev, Besov and Triebel-Lizorkin spaces on quantum tori.
		\newblock {\em Mem. Amer. Math. Soc.} 252 (2018), no. 1203, vi+118 pp.
		
		
		
		\bibitem{Xu2007} Q. Xu.
		\newblock Noncommutative $L_p$-spaces and martingale inequalities.
		\newblock (unpublished manuscript), 2007.
		
		
		
		
		
		
		
		
		
\bibitem{SZ2023} D.~V. Zanin and F.~A. Sukochev.
\newblock Connes integration formula: a constructive approach.
\newblock {\em Funct. Anal. Appl.}, 57 (2023), 40--59.







		
	\end{thebibliography}
\end{document}